\documentclass[12pt,a4paper]{article}
\usepackage{amssymb}
\usepackage{amsmath,amsfonts,amsthm}

\usepackage{graphicx}
\usepackage{amsmath}

\usepackage{amsfonts}
\usepackage{amsthm,amscd}
\usepackage{graphicx}
\usepackage{amsmath}
\usepackage{amssymb}
\usepackage{amstext}

\newtheorem{theorem}{Theorem}
\newtheorem{corollary}[theorem]{Corollary}
\newtheorem{definition}[theorem]{Definition}
\newtheorem{example}[theorem]{Example}

\newtheorem{proposition}[theorem]{Proposition}
\newtheorem{remark}[theorem]{Remark}

\title{Nonequilibrium in Thermodynamic Formalism: the Second Law, gases and Information Geometry}

\author{A. O. Lopes and R. Ruggiero}

\begin{document}

\maketitle

\begin{abstract} 
In Nonequilibrium Thermodynamics and Information Theory, the relative entropy (or, KL divergence) plays a very   important role.
Consider a  H\"older Jacobian  $J$ and the Ruelle (transfer) operator $\mathcal{L}_{\log J}.$ Two equilibrium  probabilities $\mu_1$ and $\mu_2$, can interact via a  discrete-time {\it Thermodynamic Operation} given  by the action   {\it of the dual of the Ruelle operator} 
$ \mathcal{L}_{\log J}^*$.   We argue that the law 
$\mu \to  \mathcal{L}_{\log J}^*(\mu)$, producing nonequilibrium,  can be seen as  a Thermodynamic Operation after showing that it's a manifestation of the Second Law of Thermodynamics.
We also  show that the change of  relative entropy satisfies
$$  D_{K L}      (\mu_1,\mu_2)  -  D_{K L}      (\mathcal{L}_{\log J}^*(\mu_1),\mathcal{L}_{\log J}^*(\mu_2))= 0.$$
Furthermore, we  describe sufficient conditions on $J,\mu_1$ for getting  $h(\mathcal{L}_{\log J}^*(\mu_1))\geq h(\mu_1)$, where $h$ is entropy. Recalling a natural Riemannian metric in the  Banach manifold of H\"older equilibrium probabilities  we exhibit the second-order Taylor formula for an infinitesimal tangent  change of KL divergence; a  crucial estimate in Information Geometry. We introduce concepts like heat, work, volume, pressure, and internal energy, which play here the role of the analogous ones in  Thermodynamics of gases. We briefly describe the MaxEnt method.
\end{abstract}

{\center Inst. Mat. UFRGS, Av. Bento Gonvalves 9500 - 91.500 Porto Alegre, Brazil}

Email of Artur O. Lopes is arturoscar.lopes@gmail.com

\smallskip

Email of R. Ruggiero is Rafael.O.Ruggiero@gmail.com

\section{Introduction} \label{int}

\medskip

This work has multiple purposes and can be seen as a modest attempt to increase the scope of the problems that can be considered within Thermodynamic Formalism. We tried to identify several fundamental concepts of Gas Thermodynamics and Information Theory in order to  briefly describe what are their correspondents in Thermodynamic Formalism. It can also be seen as a {\it  dictionary} or a {\it guide} for the mathematician less familiarized with certain concepts, relationships, and fundamental problems of Mathematical Physics; we describe a large range of topics that could establish a common ground for discussion between mathematicians and physicists. The reader will notice that under the scope of {\it nonequilibrium} we have gathered several topics that, in a sense, are not exactly related.  Non equilibrium is a topic that can encompass quite different aspects.
We point out that Thermodynamic Formalism is  a {\it  dynamical theory}.  We will deal mainly with H\"older Gibbs probabilities on $\Omega=\{1,2,...,d\}^\mathbb{N}$  (which are all invariant for the shift  $\sigma$) and these probabilities are singular with respect to each other -  the entropy considered here is the {\it Shannon-Kolmogorov} entropy.

Here we describe {\it idealized models of real physical systems} related to non equilibrium Thermodynamics. 
We recall a statement by  O. Penrose  (author of  ``Foundations of Statistical Mechanics'' - reference \cite{Penrose}) cited in the beginning of section 3 in \cite{Altaner1}: 

{\it The crucial [postulate in idealized models of real physical systems] is
expressing the assumption, that the successive observational states of a
physical system form a Markov chain. This is a strong assumption, [. . . ],
but even so, it has been adopted here because it provides the simplest precise
formulation of a hypothesis that appears to underlie all applications
of probability theory in physics.}

\smallskip

In Section \ref{gaz} we explain for the mathematician some of the basic results and properties of  gas Thermodynamics; in particular, what is a {\it Thermodynamic Operation} in this setting.
We point out that in Thermodynamic of gases the variation of entropy is more important than the absolute value of entropy. 
The {\bf Second Law of Thermodynamics} claims that the change of entropy must be {\bf zero} for {\bf reversible} transformations and {\bf positive} for {\bf irreversible} transformations of the system.

We advocate that a procedure taking an equilibrium probability $\mu_1$ to a non equilibrium probability $\mu_2$ qualifies to be called a
{\it Thermodynamic Operation}, if the following is true: $\mu_2$ is naturally associated with a potential $A_2$, that in turn has an equilibrium
probability  $\mu_3$, and, finally, it turns out  that the entropy of $\mu_3$ is greater than the entropy of $\mu_1$ (see Remark \ref{uau} here).

Definitions  \ref{first} and \ref{klr}  will describe two possible meanings for the {\it  Second Law of Thermodynamics for irreversible systems} in Thermodynamic Formalism.
If one wants to abstract all issues related to Physics,  we can say, in purely mathematical terms,  that our Section \ref{pre} is a study of the change of entropy and 
relative entropy under the action of the dual of the Ruelle operator. Under the {\it irreversible Second Law} the entropy should increase.

The convolution of Gibbs probabilities on the circle (invariant for the dynamics  $T(x)=2 x$, mod 1), which increase entropy, could be also be seen as a form of
Thermodynamic Operation (see \cite{Lind} and \cite{Lo2})

In Information Theory an increase in entropy is associated with an increase in the uncertainty of the information  and
this also can be seen as a discrete-time form of the Second Law of Thermodynamics.

Reading  \cite{Cat}, \cite{Altaner} and \cite{Wolp} was quite  enlightening to the understanding of  the interplay  between Statistical Physics and Information Theory. \cite{MNS} covers a large number of  interesting results in Nonequilibrium Statistical Mechanics. Sections 4.2, 4.3 and 4.8  in \cite{Cat} are dedicated to the use of KL divergence in Physics. Section E in \cite{Altaner} describes the importance of the role of KL divergence  for formulating
nonequilibrium thermodynamics (see also \cite{Shore}).
Section 4.1 in \cite{Sa} describes in the setting of Information  Theory the relation of KL divergence and the second law of Thermodynamics.

The identification of the macroscopic entropy with the lack of information about the
microstate is the basic link between probability and phenomenological thermodynamics (see \cite{Sch}, \cite{Cat}, \cite{Wang}, 
\cite{Ben}, \cite{ThoQi} and \cite{Altaner2}).
From the point of view of Physics, a self-consistent approach to modern nonequilibrium thermodynamics based on the rationale of Information theory is presented in \cite{Altaner}.

Concepts like energy, work  and entropy production are  in common use in  \cite{CL} which considers the thermodynamics of computation and error correction (see also  \cite{ChuS} which considers the computation of finite state machines and time-inhomogeneous Markov chains). Section VIII in \cite{Bennett} describes  Landauer’s principle in Information Theory, which is related to
logically reversible computations in a thermodynamically reversible fashion.

Our reasoning in Sections 2 to 4  was motivated  by  the works  \cite{Sa} and \cite{Sch}, where  the action  $p \to P\, p$, with  $P$ a $d$ by $d$  stochastic matrix and $p=(p_1,p_2,...,p_d)$ a vector of probability,  plays the role of  random source (see also \cite{Rached}). 
They show that
$$D_{KL} (p,q)     -  D_{KL} (P (p),P(q))\geq 0,$$
where $D_{KL}$ is KL divergence.
When the matrix $P$ is double stochastic they show that the entropy $h$ increase under the action $p \to P\, p$;  and we advocate the claim that this is a manifestation of the Second Law of Thermodynamics (see for instance Example \ref{exx}). The action of $P$ is the action of a Markov Chain operator (see expression (4.3) in \cite{Sa} for the setting of Information Theory  and also expressions (3.8) and (3.14) page 26 in \cite{Sch} for the setting of Thermodynamic of gases). In information Theory the action of the matrix $P$ is one of possible models for a random channel (see  Example 9.4.14 in \cite{Gray},  \cite{mui} or \cite{pou}). Given a probability $\mu$, the value $-h(\mu)=I(\mu)$  is sometimes called the {\it information measure} of  $\mu$ (see (1.11) in page 21 in \cite{Sch}). Larger the entropy smaller the information measure of $\mu$.

Here the action of the dual of the Ruelle operator (also called transfer operator) will replace the action of the stochastic matrix $P$ in \cite{Sa} and \cite{Sch}. Remark \ref {uau} in Section \ref{FR} will justify the claim that the law $\mu_1 \to \mathcal{L}_{\log J}^*(\mu_1)$
corresponds to a {\it discrete time irreversible  thermodynamic operation}.   There are some conceptual differences between the two settings and this will be explained in Remark \ref{uyt8}.

We   investigate in Section \ref{pre} the change of KL divergence and also  sufficient conditions for the increase  of entropy $h$ under the action of $\mathcal{L}_{\log J}^*$.  This will be carefully explained later in Section \ref{pre}. We point out that the KL divergence is not a  continuous function; in fact, it is  a lower semi-continuous  function of pair of probabilities (see \cite{Gray} or \cite{Posner}).

Results for the law $\mu_1 \to \mathcal{L}_{\log J}^*(\mu_1)$ also contemplates the Markov Chain (Markov probabilities) case, indeed, in  Example \ref{exlo} we explain why $e^{\log J}=J$ plays the role of the matrix $P$. 
 Section 3 in \cite{Altaner1} also consider the Markov Chain case but in a setting that we believe is different from ours.

\smallskip

 In our paper the Sections \ref{pre} and  \ref{fish} contain new results. The other sections  have mainly the purpose of describing
 the main concepts of Thermodynamic of gases in the setting of Thermodynamic Formalism.

This article is structured as follows: in section \ref{FR} we present a brief description of  some topics in Thermodynamic Formalism and the Ruelle operator. Section \ref{KL} outlines some definitions and basic  properties  of  the KL divergence.

In section \ref{pre} we   investigate the change of KL divergence under the action of the dual of the Ruelle operator.       Among other things we show that $D_{KL}\, (\mu_1, \mu_2) - D_{K L}(\mathcal{L}_{\log J}^*(\mu_1),\mathcal{L}_{\log J}^*(\mu_2) ) = 0 $ (see Theorem \ref{yyt}). We also present sufficient conditions on $J,\mu$ for the validity  
of $h(\mathcal{L}_{\log J}^*(\mu)) \geq h(\mu)$. Theorem \ref{yyt721}  will show that it makes sense to call the law $\mu \to \mathcal{L}_{\log J}^*(\mu)$ a discrete time Thermodynamical operation (see also Remark \ref{uau} in Section \ref{FR}).
The action of $\mathcal{L}_{\log J}^*$ is in accordance  with the {\it Second Law for irreversible systems}  point of view.

Section \ref{fish} considers briefly some of the main issues in Information Geometry for the case of the Riemannian Banach manifold of H\"older Gibbs probabilities: Fisher Information, susceptibility, asymptotic variance, curvature  and infinitesimal variation of the KL divergence.
This corresponds to a  continuous variation of thermodynamic quantities via a tangent vector. Among other things we are interested in
the Taylor formula of order $2$ for
\begin{equation} \label{brid1} t \to  D_{K L}      (\mu_0,\mu_t),
\end{equation}
when $\mu_t \in \mathcal{G}$ is obtained via an infinitesimal tangent variation of the point $\mu_0$ at $\mathcal{G}$  (see Proposition \ref{corre} in Section \ref{fish}).
Tangent vectors  to $\mathcal{G}$  are described in \cite{KGLM} and \cite{LM}. The infinite dimensional manifold  $\mathcal{G}$ has points of positive and points of negative curvature (see \cite{LM}). The parameter $t$ in \eqref{brid1} can be considered as time  in a {\it weakly relaxing} setting (expression \eqref{brid1} above corresponds to  expression (29) in \cite{Altaner}). For the second order Taylor formula, we obtain a  result similar to the expression (1.24) in \cite{Ama}.

Section \ref{gaz} briefly describes the basic results and relations that are well known regarding the classical thermodynamic quantities of the Thermodynamic of gases.
This section can be seen as a guide to conduct us in the search for concepts and properties that should play similar roles in Thermodynamic Formalism. This section can be skipped for the reader familiar  with gas Thermodynamics.

In section \ref{gazes}, in the context of Thermodynamic Formalism, we will present the concepts of heat, work, internal energy, volume, pressure, and some of its main relations (see for instance \eqref{p275}, \eqref{p276} and \eqref{p277}). This corresponds to a  real continuous  variation of Thermodynamic quantities with volume, temperature, and other external parameters.  We will also show in \eqref{rert}  {\it the fundamental Gibbs equation} in our setting. In our understanding, the well-known concept of topological pressure does not correspond to the concept of pressure of the Thermodynamic of gases (see expression \eqref{p575}). For the mathematician unfamiliar with some of the main topics of gas thermodynamics  we also describe in this section the {\it method of maximum entropy} on the context of Thermodynamic Formalism. The mathematics in this part of the section is not exactly new (see \cite{CL} and \cite{Lall1}) but we believe is worthwhile to   put all this in perspective.

 In section \ref{TO} we will consider the discrete-time thermodynamic operation on the system in equilibrium and we are interested in the {\it  First Law  of Thermodynamics =  conservation of energy} : $ \Delta W + \Delta Q = \Delta U ,$     where  $\Delta W$ is the change in work,  $ \Delta Q$ is the change in heat and $\Delta U$  denotes the change in the internal energy.

Section \ref{EP} is  a synthetic review of \cite{LM}. In this section we describe the entropy production produced by the reversion of direction on the lattice $\mathbb{Z}$ (which could also be seen as a kind of   thermodynamic operation).   For a given potential
$ A: \{1,2, ..., d \}^\mathbb{N} \to \mathbb {R} $, we will describe the meaning of being symmetrical with respect  to the involution kernel. This can be understood as saying that the system associated with this potential is reversible. We will show that the entropy production is zero if the potential is symmetrical in relation to the involution kernel.

We would like to thank A. Caticha, C. Maes and L. F. Guidi for very helpful conversations. We thank the referees for their helpful comments during the submission process for this paper.

\section{Preliminaries on Thermodynamic  Formalism and the Ruelle operator} \label{FR}

Now, we will explain to the reader some basic properties of the Ruelle operator and its role in Thermodynamic Formalism (a more complete description of the subject can be obtained in \cite{PP} or \cite{VO}).
The reader familiar with the topic of this  section can skip it.

A element $x$ in $\Omega=\{1,2,...,d\}^\mathbb{N}$   is denoted by $ x =(x_1,x_2,...,x_j,..), x_j\in \{1,2,...,d\}$, $j\in \mathbb{N}$.
The shift $\sigma:\Omega \to \Omega$ is the transformation such that $\sigma(x_1,x_2,...,x_j,..)= (x_2,x_3,...,x_j,..).$

Consider the metric on $\Omega$  such that $d(x,y) = d( (x_1,...,x_j,..),(y_1,...,y_j,..))= 2^{-N}$, where $N$ is the smaller $j\in \mathbb{N}$, satisfying $x_j\neq y_j.$ The space $\Omega$ is compact with such metric.

A probability $\mu$ on $\Omega$ is called $\sigma$-invariant if for any continuous function $\varphi: \Omega \to \mathbb{R}$ we have that
\begin{equation} \label{Gre}\int \varphi d \mu = \int (\varphi\circ  \sigma ) d \mu.
\end{equation}

Invariant probabilities correspond to stationary stochastic process $X_n$, $n \in \mathbb{N}$,  with values in $\{1,2,...,d\}.$ In this case $n$ means time.

From the Statistical Mechanics point of view, considering $\{1,2,...,d\}$ as a set of spins  and
$\Omega=\{1,2,...,d\}^\mathbb{N}$  as the set of strings of spins in the lattice $\mathbb{N}$, the $\sigma$-invariant probabilities on
$\Omega=\{1,2,...,d\}^\mathbb{N}$ describe the set of probabilities which are invariant by translation in the lattice. In this case, $\sigma^n: \Omega \to \Omega$, $n \in \mathbb{N}$, does not mean $n$-iteration on time.

The cylinder $\overline{r} \subset \{1,2,...,d\}^\mathbb{N}$, $r=1,2,...,d$, is the set of elements of the form $(r,x_2,x_3,...,x_j,..)$.
We call $\overline{r}$ a cylinder of size one. Note that $\sigma:\overline{r} \to \Omega$ is injective. 

\smallskip

\begin{definition} \label{rrr} Given a probability $\mu$ in $\Omega$, we say that $\mu$ satisfies  the A assumption, if for any fixed  $j$, we have the property: for a Borel 
set $B\subset \overline{j}$, if 
$\mu(\sigma (B))=0$, then, $\mu(B)=0.$
\end{definition}

Invariant probabilities satisfy the A assumption. This assumption is named non-singularity in abstract ergodic theory.

 For each
 $j\in\{1,...,d\}$, denote
$\tau_j :\Omega \to \overline{j}$, the inverse of $\sigma: \overline{j} \to \Omega$.

If $\mu$ satisfies the A assumption, then, for any fixed $j$ there exist a Radon-Nikodym derivative $f_j$, such that, 
for any Borel 
set $B\subset \overline{j}$, we get  $\mu(B)= \int_{\sigma(B)} (f_j \circ \tau_j) \, d \mu.$ If $f_j$ is strictly positive, this is equivalent to the condition:
given a continuous function
$\varphi : \Omega  \to \mathbb{R}$ which is zero outside $\overline{j}$, then
\begin{equation}  \label{were1}  \int_\Omega \varphi (\tau_j(x)) \frac{1}{ f_j(\tau_j(x))} d\mu(x)=\int_{\overline{j}} \varphi(y) d \mu(y)= \int_\Omega \varphi(y) d \mu(y).
\end{equation}

Following section 9.7 in \cite{VO} we set:

\begin{definition} Given a probability $\mu$ on $\{1,2,...,d\}^\mathbb{N}$  satisfying the A assumption, assume that for any $j$, the Radon-Nikodym derivative $f_j$ is strictly positive. For each cylinder  $\overline{j}$, $r=1,2,...,d$,
we set  
$$J_j= \frac{1}{f_j},$$ 
$J_r: \overline{j} \to \mathbb{R}$.
The function $J: \Omega \to \mathbb{R}$, such that, in each  cylinder $\overline{j}$,  $j=1,2,...,d$, is equal to $J_j$, will be called the
{\bf inverse Radon-Nikodym derivative in injective branches} (IRN for short)  of the probability $\mu$.
\end{definition}

We assume here that all
$J_j, j=1,2,...,d$, are continuous. It can be also written as 
$J_j= d \mu/(d \mu \circ \sigma)|_{\overline{j}}.$

Another way of expressing the  relationship \eqref{were1}   for the IRN $ J $ of the probability $ \mu $ is as follows:
for any
measurable  Borel set $B\subset \overline{j}$
\begin{equation}  \label{were} \int_{ \sigma(B)} d \mu  =\mu ( \sigma(B)) = \int_B J_j^{-1} d \mu,
\end{equation}
or alternatively, for any continuous function $\varphi$

\begin{equation}  \label{weret1}  \int_\Omega \varphi (\tau_j(x)) \,
J (\tau_j(x)) d\mu(x)= \int_\Omega \varphi(y) d \mu(y).
\end{equation}

\begin{definition} We will say that the probability $\mu$ on $\Omega$ is suitable, if its IRN $J$ is positive and
continuous.
\end{definition}

\begin{remark} \label{rtrt}  If $\mu$ is suitable, its IRN $J$ is a continuous function bounded away from zero. Then, the support of $\mu$ is the whole set $\Omega$.  
Indeed, if there exists an open set $B$ with zero probability, then $\sigma^{-1}(B)$ also has zero measure. Now, taking inverse images 
$\sigma^{-n}(B)$ and using the expansiveness of $\sigma$ we reach a contradiction.
\end{remark}

We will be concerned here  only  with probabilities on $\Omega $ that are suitable. In some moments it will be important to check if a given probability $\mu$ on $\Omega$ is suitable and, eventually,  if its IRN  is of H\"older class.


Consider  a suitable probability $ \mu$ on 
$ \Omega= \{1,2,...,d\}^\mathbb{N}$ and the associated IRN $J$. If the suitable probability $\mu$ is invariant for the shift, then we get the property: for each $ x =(x_1,x_2,...,x_n,..) \in \{1,2,...,d\}^\mathbb{N}$
\begin{equation} \label{Colo1} \sum_{\{y \, |\, \sigma(y)=x\}} J(y) =   \sum_{a=1}^d J(a \,x)=1.
\end{equation}

For the class of probabilities $\mu$ on $\Omega$ that we consider here to say that  $\mu$ is $\sigma$-invariant is equivalent to say that its  IRN $J$ is positive and satisfies \eqref{Colo1}.

\begin{definition}
In the  case we consider  a suitable $\sigma$-invariant probability $\mu$ we will say that the IRN  $J$ is  a {\bf Jacobian}.  If $\mu$ is not invariant we will just say that $J$ is a  IRN .
\end{definition}

In other words:  a Jacobian $J:\Omega \to (0,1)$ is a continuous positive function such that \eqref{Colo1} is true.

It is fair to say that would be better to call Jacobian the function  $J^{-1}$ but  we will keep this terminology that has been used previously in other works (the Jacobian $J$ in \cite{VO} corresponds to our $J^{-1}$).

\begin{example} \label{exlo}  Suppose that  $\mu$ is  stationary Markov probability in $\Omega=\{1,2\}^\mathbb{N}$ obtained from
a column stochastic matrix $P$ with  positive entries. Denote by $\bar{\pi} = (\pi_1,\pi_2)$ the  initial vector of probability which is invariant for $P$.  The cylinder (of size two) $\overline{i j}  \subset \{1,2\}^\mathbb{N}$,  for fixed $i,j=1,2$,  is the set of elements $x$ of the form
$x=(i,j,x_3,x_4,..)$ , $x_j \in \{1,2\}$, $j\geq 3.$  By definition
$$ \mu ( \overline{j_1,  ...,
 j_n}\, )= P_{ j_n\,j_{n-1}\,}  ... \, \, P_{ j_3\,j_2}\,P_{ j_2\,j_1} \,\pi_{j_1} .$$
In this case
$\mu( \overline{i j}) =  P_{j i}  \pi_i $.

$P$ acts on the right  on probabilities $(p_1,p_2)$ (column vectors)  and on the left on functions $(f_1,f_2)$ (line vectors).   We assumed above that
$ P \left(
\begin{array}{cc}
\pi_1 \\
\pi_2
\end{array}
\right)= \left(
\begin{array}{cc}
\pi_1 \\
\pi_2
\end{array}
\right)$ and $(1\,1) P=(1\,1).$

In this case  one can show that the H\"older Jacobian $J$ of the $\sigma$-invariant probability $\mu$ is constant in cylinders of size two. For $x$ in  the cylinder  $\overline{i j} $ we get that  
$$e^{\log J (x)}= J(x) = \frac{\pi_i \,P_{i\,j}}{\pi_j}= P_{j i}=e^{\log  P_{j i}}$$ 
(see for instance \cite{PP} or  \cite{LNotes}).  Note that for all
$i\in\{1,2\}$ we get that $J(1 i) + J (2 i)=1$. The function $J$ is H\"older in this case.

If $P$ is symmetric, then  $J(x) = P_{i j}$, for $x \in \overline{i j}$.
\end{example}

  \begin{definition} Given a continuous function $A: \Omega \to \mathbb{R}$, {\bf the Ruelle operator} $\mathcal{L}_A$ acts on the set of continuous functions $\varphi: \Omega \to \mathbb{R}$ in the following way: we say that $\psi= \mathcal{L}_A (\varphi)$, if for all $x =(x_1,x_2,...,x_n,..)\in \Omega$
  $$  \psi(x) = \mathcal{L}_A (\varphi) \, (x) =$$
  \begin{equation}\label{Rulu}  \sum_{a=1}^d e^{ A ( a,  x_1,x_2,...,x_n,..) } \,\varphi ( a,  x_1,x_2,...,x_n,..) =\sum_{a=1}^d e^{ A (\tau_a(x)) } \,\varphi (\tau_a(x)).
  \end{equation}
  \end{definition}

 In Thermodynamic Formalism the Ruelle operator plays the role of the transfer operator of  Statistical Mechanics.
 We say that the potential $A$ is normalized if $\mathcal{L}_A (1)=1$.
  Consider a H\"older Jacobian $J: \Omega \to \mathbb{R} $, in this case, $\log J$ is normalized. Indeed,
$\mathcal{L}_{\log J}(1)=1$ follows from \eqref{Colo1}. Note that the Jacobian $J$ of example \ref{exlo} is normalized.

Given a H\"older potential $A$, there exists $\lambda>0$ and H\"older positive function $\varphi:\Omega \to \mathbb{R}$, such that,
$\mathcal{L}_A (\varphi) = \lambda \varphi$ (see \cite{PP}). One can show that the positive function $J$ such that
 \begin{equation}\label{Rululi}\log J = A + \log \varphi - \log (\varphi \circ \sigma) - \log \lambda
 \end{equation}
is a H\"older normalized Jacobian.

If $A =  \log J$ depends on two coordinates, the action of the operator described in \eqref{Rulu} on a vector $ (\varphi(1),...,.\varphi(d)) \in \mathbb{R}^d$ looks like  the action  of a column stochastic matrix $P$ - acting on vectors (= functions) on the left -, indeed if $A =  \log J=\log P$
\begin{equation} \label{poo}\mathcal{L}_{\log J}  (\varphi) \, (x) = \sum_{a=1}^d  e^{ \log J ( a,  x_1)}  \,\varphi (a) = \sum_{a=1}^d  J ( a,  x_1)  \,\varphi (a).
\end{equation}

We say that an invariant probability $\mu$ is a H\"older Gibbs probability (a Gibbs probability for short) if its Radon-Nikodin derivative (a Jacobian $J$) is
a positive H\"older function. Such $\mu$ has support on the all space $\Omega$ and is ergodic. There exists a bijection between H\"older Gibbs probability $\mu$ and
H\"older Jacobians.

The study of the topological pressure, entropy, and equilibrium states in the one-dimensional lattice is the main topic of Thermodynamic Formalism (see \cite{PP} and \cite{VO}). At the beginning of Section \ref{gazes} we present a more detailed  description of the main definitions and results in such theory.
{\bf Equilibrium states} (see Definition \ref{fufu}) will be also called {\bf Gibbs states} here.

We denote by
$\mathcal{G}$ the set of H\"older Gibbs probabilities. $\mathcal{G}$ is an analytical Riemannian Banach manifold (we refer the reader to \cite{KGLM} and \cite{LR} where it is shown that $\mathcal{G}$ is an infinite dimensional analytic manifold with a natural Riemmanian metric associated to the asymptotic variance). Two different H\"older Gibbs probabilities are singular with each other.

An important property of the Ruelle operator $\mathcal{L}_{\log J}$, where $J$ is the Jacobian of an invariant probability $\mu$   is: for any continuous function $\varphi: \Omega \to \mathbb{R}$
\begin{equation} \label{Colo}   \mathcal{L}_{\log J} (\varphi \circ \sigma) = \varphi.
\end{equation}

This follows at once from \eqref{Colo1}. Note that the claim is not true if $J$ is just a IRN.

\begin{definition} Given a continuous function $A:\Omega \to \mathbb{R}$, 
the {\bf dual  Ruelle operator} $\mathcal{L}_{A}^*$ acts on  finite measures on $\Omega$ in the following way: we set that  $\rho_2 = \mathcal{L}_{A}^* (\rho_1)$, in the case that for any
continuous 
function $\varphi: \Omega \to \mathbb{R}$
\begin{equation} \label{bact3} \int \varphi d \rho_2 =   \int  \mathcal{L}_{A} (\varphi) d \rho_1.
\end{equation}
\end{definition}

 When $A= \log J$, for some continuous Jacobian $J$, then $\mathcal{L}_{\log J}^*$ acts on probabilities on $\Omega$ because $\mathcal{L}_{\log J}(1)=1$.

Given a H\"older Jacobian $J$,
the dual  Ruelle operator $\mathcal{L}_{\log J}^*$ preserves the set of probabilities but does not preserve the set of invariant probabilities. Indeed, assume that $\mu_1$ is invariant, denote
 $\mu_2 = \mathcal{L}_{\log J}^* (\mu_1)$, then we claim that $\mu_2$ is not invariant.
In fact, assume by contradiction that $\mu_2$ is invariant: given a continuous function $\varphi$, then, from \eqref{Colo}
 \begin{equation} \label{corri}\int \varphi d  \mu_2 =\int (\varphi \circ \sigma) d  \mu_2=\int\mathcal{L}_{\log J} (\varphi \circ \sigma) d  \mu_1=  \int \varphi d \mu_1.
 \end{equation}

 From this follows that $\mu_2$ is invariant just in the case $\mu_1=\mu_2$. Note that also in the  case that $ \mu_1 $ is not invariant
   we  get that $ \mu_2 $ is not invariant.

\begin{remark} \label{supe} If   $\mathcal{L}_{\log J}$ is the Ruelle operator for a continuous  Jacobian $\log J:\Omega \to \mathbb{R}$ and $\mu_1$ is a suitable probability, then, $\mu_2= \mathcal{L}_{\log J}^* (\mu_1)$ is also  suitable and has a continuous  IRN
$J_2:\Omega \to \mathbb{R}$ (see Corollary \ref{cont}). If $J$ is H\"older and the IRN of $\mu_1$ is H\"older, then $J_2$ is also H\"older.
\end{remark}

\begin{proposition}
 If $J$ is the IRN of  a suitable probability $\mu$, not necessarily invariant, then,
 \begin{equation} \label{corri2} \mathcal{L}_{\log J}^* (\mu) =\mu
 \end{equation}
\end{proposition}
\begin{proof}  We have to show that for any continuous function $\varphi:\Omega \to \mathbb{R}$ we get that
$$  \int \mathcal{L}_{\log J} (\varphi) (x) d \mu(x)  = \int \varphi d \mu.$$

For $j\in\{1,2,...,d\}$ we denote $\varphi_j$ the function $\varphi$ restricted to the cylinder $\overline{j}$ and zero outside $\overline{j}$ .
From \eqref{were1}  we get for each $j$ that
$$ \int_\Omega \varphi_j (\tau_j(x)) J(\tau_j(x)) d\mu(x)=\int_{\overline{j}} \varphi(y) d \mu(y)= \int_\Omega \varphi_j(y) d \mu(y).$$

Then,
$$  \int \mathcal{L}_{\log J} (\varphi) (x) d \mu(x) = \sum_{ j=1}^d \int \varphi_j (\tau_j(x)) J(\tau_j(x)) d \mu(x) =$$
$$       \sum_{ j=1}^d \int \varphi_j(y) d \mu(y)=   \sum_{ j=1}^d  \int_{\overline{j}} \varphi(y) d \mu(y) = \int \varphi(y) d \mu(y). $$
\end{proof}

Expression \eqref{corri2} means that $\mu$ is an eigenprobability for the Ruelle operator of the  potential $\log J$ (see \cite{PP}).

 The  {\bf Shannon-Kolmogorov entropy} of a $\sigma$-invariant probability $\mu$ which has a positive H\"older Jacobian $J$ (see Theorem 9.7.3  in \cite{VO}) is
 \begin{equation} \label{bact4} h(\mu)= - \int \log J d \mu.
 \end{equation}

For the classical definition of entropy for a general shift invariant probability  we refer the reader to \cite{Wal}, section 2.6.3 in \cite{Altaner1} or section 2.3 in \cite{ThoQi}. 

\smallskip
\begin{definition} \label{orte} If $\mu$ is not $\sigma$-invariant but has a continuous  IRN  $J$, it is natural to call $- \int \log J d \mu$ its entropy and denote this value by $h(\mu)$.
\end{definition}

\smallskip

From \eqref{bact4} we get that
the entropy of the invariant probability $\mu$ of Example \ref{exlo}   is
\begin{equation}\label{lkj} h(\mu)= -  \sum_{ i, j=1}^2  \pi_j P_{ij} \log P_{ij}.
\end{equation}




\medskip

Denote by $\mathcal{M}$ the set of $\sigma$-invariant probabilities.

\begin{definition} \label{fufu} Given a H\"older potential $A:\Omega \to \mathbb{R}$ we denote by $\mu_A$ the associated {\bf equilibrium state} for $A$, that is, the $\sigma$-invariant probability $\mu_A$ which maximizes
\begin{equation} \label{bact7} \mathfrak{P}(A)= \sup_{\mu \,\in \mathcal{M}}  \{ \int A d \mu + h(\mu)\}.
\end{equation}
\end{definition}

Given $A$, the value  $\mathfrak{P}(A)$ is called the {\it topological pressure} of $A$.
If $A$ is H\"older then the equilibrium state $\mu_A$ is unique (no phase transition).
In this case, the  Jacobian $J^A$ associated with $\mu_A$ is positive and H\"older (see expression \eqref{pois}). Therefore, any $\mu_A$ is suitable. According to our definitions, the concepts of equilibrium state and Gibbs state are equivalent. 

\smallskip

A possible meaning for the Second Law in Thermodynamic Formalism is the following:

\begin{definition} \label{first} First version - Given a H\"older Jacobian $J$ and the probability  $\mu_1$, we get  the IRN $J_2$ for the probability   
$ \mu_2=\mathcal{L}^*_{\log J} (\mu_1)$ (which is not invariant). Denote by  $\mu_3$ the equilibrium probability for the  new potential (a new energy Hamiltonian) $\log J_2$. We say the pair $(J,\mu_1)$ satisfies the Second Law (for irreversible systems)  if $h(\mu_1) \leq h(\mu_3)$.  
\end{definition}

Theorem \ref{yyt721} will show that, given any pair  $(J,\mu_1)$ as above,  the Second Law (first version) given by Definition \ref{first}  is satisfied. This result is in consonance with the reasoning   synthetically described by the  chain of steps described by expression (5.102) in page 139 on  Section 5.7 in \cite{Cat}.

\smallskip

\begin{remark} \label{uau}Given the Jacobian $J$, we claim that it is natural to  call  the law $\mu_1 \to \mathcal{L}^*_{\log J} (\mu_1)=\mu_2$
an  irreversible  thermodynamic operation. Theorem \ref{yyt721} together with Remark \ref{uut} present a solid argument in this favor.
From the proof of Theorem \ref{yyt721} (showing $h(\mu_1) \leq h(\mu_3)$) one can see that there are examples where $h(\mu_1) < h(\mu_3)$,  when     $\mu_3$ is  the equilibrium probability for  $\log J_2$; all this is   in accordance  with the ''Second Law of Thermodynamics  for irreversible systems'' point of view as described by  Definition \ref{first}.   
\end{remark}

\smallskip

If $J$ is a H\"older Jacobian of an invariant probability $\mu$, then from \cite{PP}
\begin{equation} \label{bact7}
\mathfrak{P}(\log J)=0.
\end{equation}

\begin{proposition} \label{oba} In the case $J_1$ is the H\"older IRN of a suitable probability $\mu_1$, then we also have that 
\begin{equation} \label{bact707}
\mathfrak{P}(\log J_1)=0.
\end{equation}

\end{proposition}

\begin{proof}

Take in expression \eqref{Rululi} the potential $A= \log J_1$, then there exists $\varphi$ and $\lambda$, for  the Ruelle operator 
$\mathcal{L}_{\log J_1}$, such that,
$$ \log J = \log J_1 + \log \varphi - \log (\varphi \circ \sigma) - \log \lambda$$
satisfies $\mathfrak{P} (\log J)=0.$ It follows that $\mathfrak{P} (\log J_1)=\log \lambda.$

We assume that $\varphi$ satisfies the normalization condition  $\int \varphi d \mu_1=1.$

The equilibrium probability $\mu$ for $\log J$ satisfies $\mathcal{L}_{\log J}^* (\mu)= \mu$ (and it is the unique probability satisfying this property).

We claim that $\mu= \varphi \mu_1$.
Indeed, we will show that $\mathcal{L}_{\log J}^* (\varphi\, \mu_1)=  \varphi\, \mu_1$ and from this will follow that $\mu =\varphi\, \mu_1.$

Denote $\mu_2= \varphi \mu_1$.

Given any continuous function $g:\Omega \to \mathbb{R}$, as $\mathcal{L}_{\log J_1}^* (\mu_1)=\mu_1$ from \eqref{corri2}, we get 
$$ \int g\, d \mathcal{L}_{\log J}^*(\mu_2) = \int \sum_a g(a x) J (a x) d \mu_2 (x)=  \int \sum_a g(a x) J (a x) \varphi (x) d \mu_1 (x)=$$
$$  \int \sum_a g(a x) J_1 (a x) \frac{\varphi (a x)}{\varphi(x) \lambda} \varphi (x) d \mu_1(x)=  \int \sum_a g(a x) \varphi(a x) \frac{J_1(a x)}{ \lambda}  d \mu_1 (x)=$$
\begin{equation} \label{roro} \int  \frac{ g(x)\, \varphi (x)}{ \lambda}  d \mu_1 (x) = \int  \frac{ g(x))}{ \lambda} \varphi(x) d \mu_1 (x)= \int  \frac{ g(x)}{ \lambda}  d \mu_2 (x).
\end{equation}

This shows that $\mathcal{L}_{\log J}^* (\varphi\, \mu_1)  = \frac{1}{\lambda}  \varphi\, \mu_1.$ As $\varphi\, \mu_1$ is a measure
(not a signed measure) we get that $\lambda=1$ (for more details see Remark \ref{poispois} in Section \ref{pre}).

Therefore, 
\begin{equation} \label{rewiew}\mathcal{L}_{\log J}^* (\varphi\, \mu_1)=  \varphi\, \mu_1,\text{and}\, \mu\,\text{is absolutely continuous with respect to}\,\, \mu_1.
\end{equation}

Finally, from \eqref{rewiew}
$$ h(\mu) = - \int  \log J d \mu= -\int (\log J_1 + \log \varphi - \log (\varphi \circ \sigma))\, d \mu=$$
\begin{equation} \label{rew} = -\int \log J_1 \, d \mu = -  \int \log J_1 \, \,\varphi  \,d \mu_1.
\end{equation}
\end{proof}

The Ruelle operator is the main tool for showing important dynamical properties for
the equilibrium probability like ergodicity, exponential decay of correlation, etc. (in the case the potential $A$ is H\"older).
If $A$ is  of
H\"older class the supremum in  \eqref{bact7}  can be taken over $\mathcal{G}.$

We point out that the meaning of the word {\it equilibrium} in Definition \ref{fufu} is in  the sense of the Statistical Mechanics of the one-dimensional lattice (not exactly in the sense of equilibrium for gas thermodynamics).

\smallskip

\begin{example} \label{exlopi}  Consider a column stochastic matrix $P$ with  positive entries and  a vector of probability $z=(z_1,z_2)$. 
Denote by  $\mu^z$ is  the Markov probability on $\Omega=\{1,2\}^\mathbb{N}$ obtained from the matrix
$P$ and the   initial vector of probability  $z = (z_1,z_2)$, which we assume it is {\bf not} invariant for $P$.
 Then, the Markov probability $\mu^z$ is not invariant for $\sigma$.  Anyway,
$\mu^z( \overline{i j}) =  P_{j i}  z_i $.
 In this case  one can show that the IRN $J^z$ for  $\mu^z$ is constant in cylinders of size two. For $x$ in  the cylinder  $\overline{i j} $ we get that  $J^z(x) = \frac{ \,P_{j i}   z_i}{z_j}$.
The entropy of the noninvariant probability $\mu^z$ can be estimated by
$ h(\mu^z)= - \int \log J^z d \mu^z$.

\end{example}

\smallskip

The next definition describes another  meaning for the Second Law  in the setting of Thermodynamic Formalism.

\begin{definition} \label{klr} Second version -  Given a H\"older Jacobian $J$ and a suitable probability $\mu$, we say that the pair $(J,\mu)$ satisfies 
the Second Law of Thermodynamics (for irreversible systems) if the entropy increases with the thermodynamic operation $\mu \to\mathcal{L}_{\log J}^*  (\mu)$, that is,   $h( \mathcal{L}_{\log J}^*  (\mu) ) \geq h(\mu).$

\end{definition}

It is not true that any pair  $(J,\mu)$ satisfies 
the Second Law of Thermodynamics.
In Theorem \ref{yyt7} we will present sufficient conditions for the validity of the Second Law of Thermodynamics for the pair $(J,\mu)$.

\begin{remark} \label{uyt8} There are conceptual differences between the setting we consider here, where the operation is described by 
$\mu \to \mathcal{L}_{\log J}^*(\mu)$ ($J$  a H\"older Jacobian), and the action of a $n$ by $n$  matrix $P$ on vectors of probability $p$.
Different double stochastic matrices $P$ may leave the  maximum entropy probability vector $p=(1/d,1/d,...,1/d)$ invariant.  However, if $\mu_0$ is the measure of maximal entropy for the shift $\sigma:\Omega \to \Omega$, and $\mathcal{L}_{\log J}^*(\mu_0)= \mu_0$, then $J$ is constant and equal to $1/d$. There is no other 
continuous Jacobian  $J$ with this property.

\end{remark}

\medskip

 \section{Preliminaries on KL divergence} \label{KL}

If $P $ and $P' $ are probability distributions over the same finite sample set of events $U_i$, $i=1,...,d$,
then
\begin{equation} \label{lilu}  \log \frac{ P_i}{P_i'}
\end{equation}
is the bit-number necessary to change the probability $P_i'$ into $P_i$ by a message. This
message may for instance be based on the result of a new measurement. It yields a
correction of the probability $P_i'$ into $P_i$. The mean value of the bit-number equation above
formed with the weights of the "corrected" distribution P is
\begin{equation} \label{imp1} D_{KL} (P,P')=  \sum_i P_i \log \frac{ P_i}{P_i'}  \geq 0.
\end{equation}
The above concept  is known under different names in the literature: {\it  relative entropy, KL divergence, information gain}, {\it Kullback-Leibler information} or {\it cross-entropy} (see \cite{Sch}, \cite{Gray} or \cite{Shore}). This corresponds to the mean information we get
 - from going from $P'$ to $P$ - by the knowledge of the sample $U_i$, $i=1,...,d$.

If $D_{KL} (P,P')=0$, then $P=P'$.

\smallskip

In the continuous case, given the positive densities $\varphi_1(x)$ and $\varphi_2(x)$ on $\mathbb{R}$, the KL divergence is
\begin{equation} \label{imp131} D_{KL} (\varphi_1,\varphi_2)= \int \log \frac{ \varphi_1(x)  }{ \varphi_2(x) } \varphi_1(x)   d x.
\end{equation}

\begin{remark} \label{notcon} The KL divergence is not a  continuous function but it is a lower semi-continuous  function of pairs of probabilities (see section III in \cite{Posner}).
\end{remark}

The  {\it classical point of view} used when defining KL divergence by  \eqref{imp1} and \eqref{imp131} is not exactly  a dynamical point of view.

Markov Chains, stochastic matrices and non equilibrium Thermodynamics are considered in \cite{Penrose}, in Section 3 in \cite{Altaner1} and in  Chapter II in \cite{Liggett}.  
The law $\pi \to P(\pi)$ describes a certain type of random source (see section 3.6 in \cite{Gala} or \cite{Rached}). The action of a  doubly stochastic matrix makes distributions more
random (see expression  (2.11) in \cite{Sa}) because increases the entropy $h$.

 \smallskip

  An important ingredient in our reasoning - when considering an Ergodic version of the Thermodynamic of gases - is the concept of {\bf discrete time thermodynamical operation}. We will consider a specific one given by  $\mu_1 \to \mu_2 = \mathcal{L}_{\log J}^* (\mu_1)$, where $J$ is a H\"older  Jacobian.  From \eqref{corri} it is known that the action of the operator $\mathcal{L}_{\log J}^* $ takes  $\mu_1$ out  of equilibrium. The introduction of such thermodynamical operation in our  reasoning is in consonance with (4.5) in \cite{Sa} (see also (3.6) in \cite{Sch}) and the Second Law of Thermodynamics (to be discussed in Section \ref{pre}). Example \ref{exlo} in Section \ref{FR} and the Remark \ref{uau} strongly support this claim.
  
T.  Sagawa   in \cite{Sa} analyze properties of the so called {\it average entropy production}:

 \begin{equation} \label{kwe}  D_{KL} (\pi,\bar{\pi})     -  D_{KL} (P (\pi),\bar{\pi})\geq 0,
 \end{equation}
where $P(\bar{\pi})=\bar{\pi}$.

The action $\pi \to P (\pi)$ models a random source
and the expression
\eqref{kwe} is called the {\it data processing
inequality}  according to \cite{Sa}.

In  \cite{Sa} it is also shown in expression (2.10) that given probability vectors $p,q$

 \begin{equation} \label{kwee}  D_{KL} (p,q)     -  D_{KL} (P (p),P(q))\geq 0.
 \end{equation}
 
The proof that  \eqref{kwee}  is  bigger or equal to zero  follows from
  (2.10) in \cite{Sa}, Proposition 4.2 in Chapter II in \cite{Liggett}, or      (3.9) and (3.10) in \cite{Sch}.  We would like to analyze the 
  left-hand side of the above inequality in our setting.

 In the dynamical setting the  {\it Kullback-Leibler divergence} is given by
 \begin{equation} \label{imp2}    D_{K L}      (\mu_1,\mu_2) =h (\mu_1,\mu_2) = - \int \log J_2 d \mu_1 + \int  \log J_1 d \mu_1,
 \end{equation}
 where $\mu_2$ has IRN $J_2$ and $\mu_1$ has IRN $J_1$. 
 
 When $J_1$ is a H\"older Jacobian and $J_2$ a suitable H\"older IRN  one can show that 
 \eqref{imp2}   is greater or equal to zero. Indeed, from
 \eqref{bact707} we get that
$\mathfrak{P}(\log J_2)=0$. Then, it follows from \eqref{bact4} and \eqref{bact707} that
$$   \int \log J_2 d \mu_1 - \int  \log J_1 d \mu_1 =$$
\begin{equation} \label{impas2} \int \log J_2 d \mu_1 + h(\mu_1)     \leq\sup_{\mu \,\in \mathcal{M}}  \{ \int \log J_2 d \mu + h(\mu)\}    =  \mathfrak{P}(\log J_2)=0.
\end{equation}
 
 Note if $\log J_2$ is a H\"older Jacobian and $\mu_1$ is shift invariant, then by uniqueness of the equilibrium state, 
 \begin{equation} \label{impor}  D_{KL}(\mu_1,\mu_2)=0 \,\,
\Leftrightarrow\,\,
 \mu_1=\mu_2.
 \end{equation}

 For general properties of the  Kullback-Leibler divergence in the dynamic setting  see for instance \cite {Cha2} or
 \cite{LM} (when the alphabet is a compact metric space). The work \cite{ACR} describes the relation of  Kullback-Leibler divergence and \cite{LMMS} with 
 the classical concept of specific entropy in Statistical Mechanics as presented  in \cite{Geo} (see also Proposition 40 in \cite{LM}).

 \begin{remark}  We emphasize the fact that on the way to compare properties referring to concepts like \eqref {imp1} to analogous properties referring to the dynamic case \eqref {imp2} there is a notable difference: when $ \mu_1$ and $\mu_2$ are different H\"older Gibbs probabilities they
are singular to each other.

\end{remark}
 Note that when $\mu_2$ is the measure of maximal entropy we get that
 $$ h(\mu_1) =  \log d  - D_{K L}      (\mu_1,\mu_2)  .$$

\begin{remark} \label{poii} The main conceptual difference of the present  definition of KL divergence  when compared to the classic case considered in the literature (where the two probabilities are absolutely continuous with respect  to a given measure which was  fixed {\it a priori}, like the case expressed  by \eqref{imp131} and \eqref{imp1}) is that  in \eqref{imp2} we have to rely on the IRN of the two probabilities to estimate the KL divergence.  In this case, it is clear that the concept has a dynamic component.
 \end{remark}

 It is known that if $J_2$ (see for instance \cite{KLS}) is a H\"older Jacobian for the invariant probability $\mu_2$, then, given any probability $\mu_1$ on $\Omega$, we have that
  \begin{equation} \label{otimo1}\lim_{n\to \infty} (\mathcal{L}_{\log J_2}^*)^n  \,(\mu_1)=\mu_2.
 \end{equation}

 Here, among other things we are interested in  studying the expression
 $$   \text{ep}^{dyn} (\mu_1,\mu_2) :=   D_{K L}      (\mu_1,\mu_2)  -  D_{K L}      (\mathcal{L}_{\log J_2}^*(\mu_1),\mu_2),$$
which should be called the  {\it dynamical entropy production} for $J_2$.

Under the assumption that $\mu_2$ is $\sigma$-invariant with a H\"older Jacobian  $J_2$, but  $\mu_1$  not necessarily invariant,
in Theorem \ref{yyt99} in Section \ref{pre} we will show that
$\text{ep}^{dyn} (\mu_1,\mu_2) = 0.$ This equality can be interpreted as saying that
$\mathcal{L}_{\log J_2}^*$  maintains the  KL divergence when acting on the first variable.

The estimate of $   D_{K L}      (\mu_1,\mu_2)  -  D_{K L}      (\,(\mathcal{L}_{\log J_2}^*)^n(\mu_1),\mu_2)$ is considered in 
Theorem \ref{yyt99}.

When $J$ is a H\"older Jacobian, we will also analyze  the expression
 $$   \text{cg}^{dyn} (\mu_1,\mu_2) :=   D_{K L}      (\mu_1,\mu_2)  -  D_{K L}      (\mathcal{L}_{\log J}^*(\mu_1),
 \mathcal{L}_{\log J}^*(\mu_2)),$$
 which should be called the  {\it dynamical coarse-grained entropy production} for $J$.
In Theorem \ref{yyt} in Section \ref{pre} we will show that
$\text{cg}^{dyn} (\mu_1,\mu_2) = 0$, when $\mu_1$ and $\mu_2$ are suitable. We will not assume that either $\mu_1$ or $\mu_2$ is invariant.
 This equality can be interpreted as saying that
$\mathcal{L}_{\log J}^*$  maintains the KL divergence when acting on the both  variables. The equality $\text{cg}^{dyn} (\mu_1,\mu_2) = 0$ corresponds to expression (2.12) in page 25 in \cite{Sch} which considers
KL divergence and change of coordinates (see also (4.46) in \cite{Cat}).

\begin{remark} \label{ppo} From  the property $D_{K L}      (\mu_1,\mu_2)  =  D_{K L}      (\mathcal{L}_{\log J}^*(\mu_1),
\mathcal{L}_{\log J}^*(\mu_2))$
it follows that $D_{K L}      (\mu_1,\mu_2)  =  D_{K L}      (\,(\mathcal{L}_{\log J}^*)^n(\mu_1),(\mathcal{L}_{\log J}^*)^n(\mu_2\,))$. From this last expression, taking the limit in $n \to \infty$, it might seem to the reader that \eqref{impor}  and \eqref{otimo1}  could lead to a contradiction. But this does not really happen because the function $(\rho_1,\rho_2) \to D_{KL} (\rho_1,\rho_2)$ is not continuous (see Remark \ref{notcon})
\end{remark}



\section{The action of the dual of the Ruelle operator and the Second Law} \label{pre}

By definition the KL divergence of the pair $(\mu_1,\mu_2)$ is
\begin{equation} \label{imp244}    D_{K L}      (\mu_1,\mu_2) =h (\mu_1,\mu_2) = - \int \log J_2 d \mu_1 + \int  \log J_1 d \mu_1.
 \end{equation}

According to \eqref{impas2}, when $J_1$ is a H\"older Jacobian and $J_2$ a suitable H\"older IRN for $\mu_2$ we get that  \eqref{imp244}   is non negative.

In this section, we will perform a certain discrete-time thermodynamic operation on the system in equilibrium and we analyze the change in KL divergence and also in entropy. We only consider probabilities on  $\Omega $ that are suitable. First, we want to investigate the change of KL divergence under the action of the dual of the Ruelle operator.

One of our main results in this section is:

\begin{theorem}  \label{yyt}   Assume  that $J$ is a H\"older Jacobian   and  $\mu_1$  and  $\mu_2$ are suitable probabilities with continuous IRN, respectively, $J_1$ and $J_2$. Then,
\begin{equation}  \label{vare12}    D_{K L}      (\mu_1,\mu_2)  -  D_{K L}      (\mathcal{L}_{\log J}^*(\mu_1),\mathcal{L}_{\log J}^*(\mu_2) )=0.
\end{equation}

\end{theorem}

\smallskip

This result will be proved later.

\smallskip

The law $\mu_1 \to \mathcal{L}_{\log J}^*(\mu_1)$ can be seen as a random source of information (see section 3.6 in \cite{Gala} for the Markov chain case). 

Expression (2.12) in \cite{Sch} claims invariance of KL divergence under change of coordinates. Expression \eqref{vare12}  describes a kind of change of coordinates.

A natural question is the following: given $J$ and $\mu_1$, we will get via the thermodynamic operation the IRN $J_2$ for the probability   
$ \mu_2=\mathcal{L}^*_{\log J} (\mu_1)$ (which is not invariant, that is, not in equilibrium).  In this way, we get a new potential (a new energy Hamiltonian) $\log J_2$. Then, we would like to know the properties of the probability $\mu_3$ which is the equilibrium probability for the potential $\log J_2$. Theorem \ref{yyt721}  will address this question: $h(\mu_1) \leq h(\mu_3)$. This will mean to show the Second Law in the sense of the  first version (see Definition \ref{first}).

We will also investigate conditions for the increase of Kolmogorov-Shannon entropy in Theorem \ref{yyt7}. This will provide conditions for  the Second Law in the sense of the second version (see Definition \ref{klr}). In Example \ref{exx} the entropy strictly increases. Expression \eqref{vare}  means the  increase of the  uncertainty of information when the source is applied.

\begin{proposition}  \label{cont} Assume  that $J$ is the H\"older Jacobian of an invariant probability $\mu$ and the suitable probability $\mu_1$  has continuous IRN $J_1$. Then,
\begin{equation}\label{loloqati} J_3(z) = \frac{J_1(\sigma(z))\,J(z)}{J(\sigma(z))},
 \end{equation}
is the continuous  IRN of the suitable probability $\mu_3= \mathcal{L}^*_{\log J} (\mu_1).$ If $J$ and $J_1$ are H\"older, then $J_3$ is also H\"older.
\end{proposition}

\begin{proof}  Suppose $\varphi$ is zero outside the cylinder $\overline{j}$.

We have to show that (see \eqref{weret1})
\begin{equation}\label{lolopi}\int \varphi (y) d \mu_3 (y) = \int \varphi (j y) J_3 ( j y) d \mu_3(y).
\end{equation}

 Note that
$$ \int \varphi (y) d \mu_3 (y) =\int \varphi (y) d \mathcal{L}_{\log J}^*(\mu_1)(y) =$$
\begin{equation}\label{lolopu}\int \sum_a  J(a x) \varphi ( a x) d \mu_1 (x) = \int   J(j x) \varphi ( j x) d \mu_1 (x).
\end{equation}

On the other hand, from \eqref{loloqati} and \eqref{corri2} 
\begin{equation}\label{lolopulou}  \int \varphi (j y) J_3 ( j y) d \mu_3(y)= \int \sum_a  J (a x)  \varphi (j  a x) J_3 ( j a x) d \mu_1(x) =\end{equation}

\begin{equation}\label{lolopulou1} \int \sum_a  J (a x)  \varphi (j  a x) \frac{J_1 (a x) J(j a x)}{J (a x)   }   d \mu_1(x) =
\end{equation}
\begin{equation}\label{lolopulou2}   \int \sum_a  J_1 (a x)  \varphi (j  a x) J(j a x)  d \mu_1(x) =    \int  J(j  x) \varphi (j   x) d \mu_1(x) .
\end{equation}

In the last equality we use the  property $\mathcal{L}_{\log J_1}^{*} (\mu_1)=\mu_1.$

Therefore,  \eqref{lolopi} is true. $J_3$ is positive and  continuous because is the composition, product, and quotient of {\bf positive} continuous functions.

If $J$ and $J_1$ are H\"older continuous, then 
$J_3$ is   H\"older continuous because is the composition, product, and quotient of H\"older continuous {\bf positive} functions.
\end{proof}

Expression \eqref{loloqati} corresponds in some sense to the equality condition  after expression  (4.3) in Proposition 4.2 in Chapter II in  \cite{Liggett}.

Now we will present the proof of Theorem \ref{yyt}.

\begin{proof}  
We want to show  that
$$D_{K L}      (\mu_1,\mu_2)  -  D_{K L}      (\mathcal{L}_{\log J}^*(\mu_1),\mathcal{L}_{\log J}^*(\mu_2) ) =   0.$$

Indeed, denote $J_3$ the continuous IRN of $\mu_3=\mathcal{L}_{\log J}^*(\mu_1)$,  $J_4$ the continuous IRN of $\mu_4=\mathcal{L}_{\log J}^*(\mu_2)$, $J_1$ the IRN of $\mu_1$ and $J_2$ the IRN of $\mu_2$.

Then, from \eqref{loloqati}
$$    D_{K L}      (\mu_1,\mu_2)  -  D_{K L}      (\mathcal{L}_{\log J}^*(\mu_1),\mathcal{L}_{\log J}^*(\mu_2) )  = $$
$$[- \int \log J_2 d \mu_1 + \int  \log J_1 d \mu_1 ] - [- \int \log J_4 d \mathcal{L}_{\log J}^*(\mu_1) + \int  \log J_3 \,d \mathcal{L}_{\log J}^*(\mu_1)]=$$
$$- \int \log J_2 d \mu_1 + \int  \log J_1 d \mu_1  +$$
$$ \int \sum_a J(a x) \log J_4 (a x)  d \mu_1 (x) - \int   \sum_a J(a x)  \log J_3 (ax)\,d \mu_1 (x)=$$
$$- \int \sum_a J(a x)   \log J_2 (x)d \mu_1(x) + \int  \sum_a J(a x)  \log J_1 (x) d \mu_1(x)  + $$
$$\int \sum_a J(a x) \log J_4 (a x)  d \mu_1(x) - \int   \sum_a J(a x)  \log J_3 (ax)\,d \mu_1 (x)=$$
$$  \int \sum_a J(a x) \log   \frac{ J_1(x)\, J_4(a x)}{J_2(x) \, J_3(a x)}   d \mu_1(x)=$$
$$  \int \sum_a J(a x) \log   \frac{ J_1(x)\, \frac{J_2 (x)\,J(a x)}{J(x) }}{J_2(x) \, \frac{J_1 (x)\,J(a x)}{J(x) }}   d \mu_1(x)=$$ 
$$  \int \sum_a J(a x) \log   \frac{ J_1(x)\, J_2 (x)}{J_2(x) \, J_1 (x)}   d \mu_1(x)=0.$$

\end{proof}

\smallskip

The claim of the above Theorem means that the KL-divergence does not change  when $\mathcal{L}_{\log J_2}^*$  acts on pairs of probabilities.

\begin{theorem}  \label{yyt99}   Assume  that $J_2$ is a H\"older Jacobian  for the Gibbs probability $\mu_2$  and  $\mu_1$ has continuous  IRN $J_1$. Then,
\begin{equation}  \label{vare}    D_{K L}      (\mu_1,\mu_2)  -  D_{K L}      (\mathcal{L}_{\log J_2}^*(\mu_1),\mu_2)=0.
\end{equation}
\end{theorem}

\begin{proof} In Theorem \ref{yyt} take $J=J_2$, then, $\mathcal{L}_{\log J_2}^*(\mu_2)=\mu_2$ and the claim follows.

\end{proof}

\begin{theorem}  \label{yyt99}   Assume  that $J$ is the H\"older Jacobian  for the Gibbs probability $\mu$  and  $\mu_1$ has continuous  IRN $J_1$. Denote for each $n \in \mathbb{N}$, $\mu_n = (\,(\mathcal{L}_{\log J})^*\,)^n(\mu_0)$ and $J_n$ the corresponding IRN.

Then, for each $n$
\begin{equation}  \label{vareta}         D_{K L}      (\mu_n,\mu)=-\int  \log J (\sigma^n (z))  \, d \mu_n + \int \,\log   J_0 (\sigma^n(z))\,d \mu_n.
\end{equation}

\end{theorem}

\begin{proof} From \eqref{loloqati}  we get 
$$J_1(z)= \frac{J_0(\sigma(z))\,J(z)}{J(\sigma(z))}$$ and, therefore,
$$  D_{K L}      (\mu_1,\mu)=   D_{K L}      ((\mathcal{L}_{\log J})^*\,(\mu_0),\mu)=- \int \log J d \mu_1 + \int \log J_1 d \mu_1 =$$
$$- \int \log J(z) d \mu_1(z) + \int \log   \frac{J_0(\sigma(z))\,J(z)}{J(\sigma(z))}  d \mu_1 (z)=$$ 
$$\int [\,\log   J_0 (\sigma(z))\,  - \log J(\sigma(z)) \,] d \mu_1 .$$

This was the case $n=1$.

When $n=2$, we get  from \eqref{loloqati} 
$$  D_{K L}      (\mu_2,\mu)=   D_{K L}      ((\mathcal{L}_{\log J})^*\,(\mu_0),\mu)=- \int \log J d \mu_2 + \int \log J_2 d \mu_2 =$$
$$- \int \log J(z) d \mu_2 (z) + \int \log   \frac{J_1(\sigma(z))\,J(z)}{J(\sigma(z))}  d \mu_2 (z)=$$ 
$$\int [\,\log   J_1(\sigma(z))\,  - \log J(\sigma(z)) \,] d \mu_2=$$
$$\int [\,\log   J_0 (\sigma^2(z))\,+ \log J (\sigma(z)  - \log J (\sigma^2 (z))  - \log J(\sigma(z)) \,] d \mu_2 =$$
$$\int [\,\log   J_0 (\sigma^2(z))\, - \log J (\sigma^2 (z))  \,] d \mu_2.$$

If the $J_n$ is the RNI of  $((\mathcal{L}_{\log J})^*)^n\,(\mu_0)$, then, from  \eqref{loloqati} we get 
\begin{equation}
J_n(z) = \frac{J_{n-1} (\sigma(z))\,J(z)}{J(\sigma(z))}.
\end{equation}

We claim that 
\begin{equation} \label{eeh}
J_n(z) = \frac{J_0 (\sigma^{n}(z))\,\,J(z)\,}{J(\sigma^{n} (z))}.
\end{equation}

The proof is by induction.

Indeed, for $n=1, 2$ is true, and if $J_{n-1} (z) = \frac{J_0 (\sigma^{n-1}(z))\,\,J(z)\,}{J(\sigma^{n-1} (z))}$,
then,

\begin{equation}
J_n(z) = \frac{J_{n-1} (\sigma(z))\,J(z)}{J(\sigma(z))}=  \frac{ \frac{J_0 (\sigma^{n}(z))\,\,J(\sigma(z))\,}{J(\sigma^{n} (z))}  }{J(\sigma(z)) }\, J(z)= \frac{J_0 (\sigma^{n}(z))\,\,J(z)\,}{J(\sigma^{n} (z))} .
\end{equation}

Therefore,
$$D_{K L}      (\mu_n,\mu) = - \int \log J d \mu_n + \int \log J_n d \mu_n =$$
$$\int [\,\log   J_0 (\sigma^n(z))\, - \log J (\sigma^n (z))  \,] d \mu_n.$$
\end{proof}


\smallskip

Note that if $J$ and the IRN  $J_1$ of $\mu_1$ are both H\"older, then, $J_2$ is H\"older (we assumed that all probabilities are suitable). The next theorem  describes a reasoning which is similar to the one described by 
Remark \ref{uut} in Section \ref{gaz} (the setting of Thermodynamic of gases).

\smallskip

\begin{theorem}  \label{yyt721} Assume that $\mu_1$ is a probability  with IRN  $J_1$, $J$ a  Jacobian (of a $\sigma$-invariant     probability $\mu$),  $J_2$  the H\"older IRN of $ \mu_2=\mathcal{L}^*_{\log J} (\mu_1)$, and $\mu_3$ the Gibbs equilibrium probability for the potential $\log J_2$. Then, $\mu_3$ is absolutely continuous with respect to $\mu_2$, more precisely, $\mu_3 = \varphi \mu_2$, where $\varphi$ is the main eigenfunction of the Ruelle operator 
$\mathcal{L}_{\log J_2}$. In addition,
the topological pressure of $\log J_2$ is equal to zero and
\begin{equation} \label{bela} h(\mu_3)=- \int \log J_2 d \mu_3 = -\int \log J_1 d \mu_3 =-\int \log J_2 \,\,\varphi \,d \mu_2.
\end{equation}

Moreover, $h(\mu_1) \leq h(\mu_3)$.

\end{theorem}

\begin{proof}

If $\varphi$ is the main eigenfunction and $\lambda$ the associated eigenvalue of the Ruelle operator 
$\mathcal{L}_{\log J_2}$, then
$$ \log J_3 = \log J_2 + \log \varphi - \log (\varphi \circ \sigma) - \log \lambda.$$

We assume that $\varphi$ satisfies a normalization condition  $\int \varphi d \mu_2=1.$

The equilibrium probability $\mu_3$ for $\log J_3$ satisfies $\mathcal{L}_{\log J_3} (\mu_3)= \mu_3$ (and it is the unique probability satisfying this property).
We claim that $\mu_3= \varphi \mu_2$.
Indeed, we will show that $\mathcal{L}_{\log J_3}^* (\varphi\, \mu_2)=  \varphi\, \mu_2$ and from this will follow that $\mu_3 =\varphi\, \mu_2.$

Given any continuous function $g:\Omega \to \mathbb{R}$, as $\mathcal{L}_{\log J_2}^* (\mu_2)=\mu_2$ from \eqref{corri2}, we get 
$$ \int g\, d \mathcal{L}_{\log J_3}^*(\mu_3) = \int \sum_a g(a x) J_3 (a x)  \mu_3 (x)=  \int \sum_a g(a x) J_3 (a x) d \mu_3 (x)=$$
$$  \int \sum_a g(a x) J_2 (a x) \frac{\varphi (a x)}{\varphi(x) \lambda} \varphi (x) d \mu_2 (x)=  \int \sum_a g(a x) J_2 (a x) \frac{\varphi (a x)}{ \lambda}  d \mu_2 (x)=$$
$$ \int \mathcal{L}_{\log J_2}   \frac{ g(x)\, \varphi (x)}{ \lambda} d \mu_2 (x)  =    \int  \frac{ g(x)\, \varphi (x)}{ \lambda}  \mathcal{L}_{\log J_2}^*(\mu_2) = $$
\begin{equation} \label{roro} \int  \frac{ g(x)\, \varphi (x)}{ \lambda}  d \mu_2 (x) = \int  \frac{ g(x))}{ \lambda} \varphi(x) d \mu_2 (x)= \int  \frac{ g(x))}{ \lambda}  d \mu_3 (x).
\end{equation}

This shows that 
\begin{equation} \label{aroro}\mathcal{L}_{\log J_3}^* (\varphi\, \mu_2)  = \frac{1}{\lambda}  \varphi\, \mu_2.
\end{equation}

We claim that if
$\rho= \varphi\, \mu_2$ is a finite measure
(not a signed measure)  and  \eqref{aroro} is true, then it follows  $\lambda=1$.

\begin{remark} \label{poispois} Indeed, the claim follows from a simple reasoning using the involution kernel (see Definition \ref{poupou} in Section \ref{EP}) as described in \cite{GLP}. Indeed, from the involution kernel and $\rho$ one can get (see \eqref{tritri})  a positive eigenfunction associated to the Ruelle operator $\mathcal{L}_{\log J_3}$, but this impossible if $\lambda\neq 1$ (see Theorem 2.2 in \cite{PP} or Proposition 12 in \cite{LNotes}).
The full details of the relation of eigenprobabilities and the eigenfunctions  via the involution kernel is explained with details in  Remark \ref{poix} and expression \eqref{tritri}.
\end{remark}

\smallskip
From the claim we get $\mathcal{L}_{\log J_3}^* (\varphi\, \mu_2)=  \varphi\, \mu_2$.

Finally, 
$$ h(\mu_3) = - \int  \log J_3 d \mu_3 = -\int (\log J_2 + \log \varphi - \log (\varphi \circ \sigma) - \log \lambda   )\, d \mu_3=$$
\begin{equation} \label{rew} = -\int \log J_2 \, d \mu_3 = -  \int \log J_2 \, \,\varphi  \,d \mu_2.
\end{equation}

Note also that from \eqref{loloqati} and the invariance of $\mu_3$ we get
$$- \int \log J_2 d \mu_3 =- \int (\log (J_1 \circ\sigma )+ \log J - \log (J \circ \sigma)\,) d \mu_3  = -\int \log J_1 d \mu_3.$$

Finally, as the $\mathfrak{P}( \log J_1) =0$, from expression \eqref{bela} and the fact that $\mu_3$ is invariant we get
$$h(\mu_1)  - h(\mu_3) =  h(\mu_1) + \int \log J_1 d \mu_3\leq 0.$$

Therefore,
$h(\mu_1)  \leq h(\mu_3)$.

If $\mu_1$ is invariant with H\"older Jacobian $J_1$, then, from uniqueness of the equilibrium state, we get that $h(\mu_1)  < h(\mu_3)$.

\end{proof}

We recall our Definition \ref{klr}    for the validity of   the Second Law of Thermodynamics (second version) for the pair $(J,\mu)$: the entropy increase with the thermodynamic operation $\mu \to\mathcal{L}_{\log J}^*  (\mu)$.

It is not true that any pair  $(J,\mu)$ satisfies 
the Second Law of Thermodynamics.
In Theorem \ref{yyt7} we will present sufficient conditions for the validity of Second Law of Thermodynamics (second version) for the pair $(J,\mu)$.

\medskip

We will show in the next theorem  that
for a certain class of Jacobians $J$ fulfilling \eqref{rrty}, the thermodynamic operations of the kind $\mu_1 \to \mathcal{L}^*_{\log J} (\mu_1)$, satisfy the second law of Thermodynamics when acting on a certain family of probabilities $\mu_1$.
It is easy to see that the Jacobian $J$ associated with the maximal entropy measure satisfies  the property \eqref{rrty} for any probability $\mu_1$.

Note that assumptions (see (2.11)  in \cite{Sa} or (3.14) in \cite{Sch}) are required for the law $p \to P\, p$ to satisfy the property of increasing entropy ($P$ is taken as  a double stochastic matrix). There are examples of stochastic matrices $P$ such that for some vector of probability $p$ its action decrease entropy. Example \ref{exx}, which considers a Jacobian given by a symmetric matrix (a particular case of  a double stochastic matrix), 
will present a case where expression \eqref{rrty} is true in a strict sense and therefore the entropy strictly increases.

\begin{theorem}  \label{yyt7} Assume that $\mu_1$ is a probability and $J$ a  Jacobian (of a $\sigma$-invariant     probability $\mu$). Denote
by $J_2$  the IRN of $ \mu_2=\mathcal{L}^*_{\log J} (\mu_1)$  and by $J_1$  the IRN of $ \mu_1$.

If
\begin{equation} \label{rrty} 1 - \int \sum_a \frac{J(a x)^2}{J (x)}\, d \mu_1(x)\geq 0,
\end{equation}
then, the pair $(J,\mu_1)$ satisfies the Second Law of Thermodynamics, that is,
\begin{equation}  \label{vare}  h(\mathcal{L}^*_{\log J} (\mu_1))= h(\mu_2)=  - \int \log J_2 d \mu_2 \geq - \int \log J_1 d \mu_1 = h(\mu_1).
\end{equation}

\end{theorem}

\begin{proof}

We will show that under condition \eqref{rrty} we get that
 \begin{equation} \label{htps} - \int \log J_2 d \mu_2 + \int \log J_1 d \mu_1 \geq 0.
 \end{equation}

As $J$ is a Jacobian we get that for all $x$ we have $\sum_a J(a x) =1.$
Then, using the inequality $   \log x \leq  1 -\frac{1}{x} $ we get

$$ \int \log J_1 d \mu_1 - \int  \log J_2 \,d \mathcal{L}_{\log J}^*(\mu_1)=$$
$$ \int  \sum_a J(a x)  \log J_1 (x) \,d \mu_1(x) -  \int \sum_a J(a x) \log J_2 (a x)  d \mu_1 (x) =$$
$$ \int \sum_a J(a x)    \log \frac{J_1 ( x)}{J_2 (a x)} d \mu_1(x)\geq   \int \sum_a J(a x) [ 1-  \frac{J_2 (a x)}{J_1 ( x)}   ]   d \mu_1(x)  =$$
$$ 1 -    \int \sum_a J(a x)   \frac{J_2 (a x)}{J_1 ( x)}     d \mu_1(x) =  1 -  \int \sum_a J_2(a x)  \frac{J(ax)}{J_1(x)}\,  d \mu_1(x) = $$
$$   1 -   \int \sum_a \frac{J_1(x)\,J(a x)}{J(x)} \frac{J(a x)}{J_1(x)}\,   d \mu_1(x)   =   1-    \int \sum_a J (a x)  \,  \frac{J (a x)}{J ( x)} d \mu_1(x) \geq 0,       $$
where above we used  \eqref{loloqati}  and in the last inequality we used the hypothesis  \eqref{rrty}.

\smallskip



\smallskip
\end{proof}

\smallskip

The next example describes a kind of dynamical version of the claim (2.11) in \cite {Sa} (about increasing entropy for the case of the action of a double stochastic $P$).

\begin{example} \label{exx} Assume that  $\Omega =\{1,2\}^\mathbb{N}$, $J$ is the Jacobian of Example \ref{exlo}, which considers a stochastic matrix $P$, and, moreover, that  $P$ is symmetric.  Then, expression \eqref{rrty} is true for any independent probability $\mu_1$, with weights $p=(p_1,p_2)$, $p_1,p_2>0$. Indeed,
$$ \int \sum_a \frac{J(a x)^2}{J (x)}\, d \mu_1(x) = \sum_{j,k} \sum_a \frac{J(a j)^2}{J (j k)}\,  \mu_1 (\overline{jk})= $$

$$
\sum_{j,k} \sum_a \frac{J(a j)^2}{J (j k)}\,   p_j=\sum_{j,k} \sum_a \frac{P_{a j}^2}{P_{j k} }\,  p_j = \sum_{j} \sum_{a} \sum_k \frac{P_{a j}^2}{P_{j k} }\,  p_j >$$
$$\sum_{j} \sum_a \sum_{k=a}  \frac{P_{a j}^2}{P_{j k} }\,  p_j= \sum_{j}  \sum_a \frac{P_{a j}^2}{P_{j a} }\,  p_j =\sum_{a}  \sum_j \frac{P_{a j}^2}{P_{ a j} }\,  p_j=$$
$$ \sum_{a}  \sum_j P_{a j}\,  p_j= \sum_{a}  \,  q_a=1,$$
where $ P(p_1,p_2) = (q_1,q_2).$

Therefore, $ h( \mathcal{L}^*_{\log J} (\mu_1))>    h(\mu_1) = h(p).$

\end{example}

\smallskip

\section{Information Geometry for Gibbs measures} \label{fish}

In this section, we consider a continuous-time variation of  thermodynamic quantities.

We denote by $\mathcal{G}$ the infinite dimensional manifold of Holder Gibbs probabilities with the associated Riemannian structure as described in \cite{KGLM}. We denote a tangent vector to $\mathcal{G}$ in the point $\mu_1$
by $\xi$. The Holder Gibbs probability probability $\mu_1 + d \xi$ is obtained by the local exponential map on the tangent plane at $\mu_1$.

Our main goal is to show

\begin{theorem} \label{corre27} If $\xi$ is a tangent vector to $\mathcal{G}$ at $\mu_1$, then,
\begin{equation} \label{china244}  D_{K L}      (\mu_1,\mu_1 + d \xi) =   \frac{1}{2} \int \xi^2 d \mu_1  +
o (|d \xi|^2) ,
\end{equation}
where $ \int \xi^2 d \mu_1$ is the Fisher  information.
\end{theorem}

 We will explain later the meaning of the expression $D_{K L}      (\mu,\mu+ d \xi) .$

Consider the H\"older potential $\log J$ and the associated probability $\mu \in \mathcal{G}$. Moreover, consider another H\"older function (a tangent vector) $\xi$ (which is fixed) and finally the potential
$A_\theta =\log J +\, \theta \,\xi$, $\theta \in \mathbb{R}$ (which is not normalized). We denote by $\log J_\theta $ the normalized potential  associated to the equilibrium probability
$\mu_\theta = \mu_{\log J +\, \theta \,\xi}$, $\theta \in \mathbb{R}$.

The Ruelle operator $\mathcal{L}_{A_\theta}$  has a positive H\"older  eigenfunction $\varphi_\theta$ and a positive eigenvalue $\lambda_\theta$. From \cite{PP} the Jacobian $\log J_\theta $ of $\mu_\theta$ can be described in the form
\begin{equation} \label{pois}\log J_\theta= A_\theta+   \log \varphi_\theta - \log \varphi_\theta (\sigma) - \log \lambda_\theta.
\end{equation}

The limit
$$\lim_{\theta \to 0} \mu_\theta =\mu$$
is a form of continuous time convergence to equilibrium.  We want to consider the derivative in the direction  $\xi$. It is natural to consider the directional derivative $\log J +\, \theta \,\xi$, for a small parameter $\theta \sim 0$.

The meaning of \eqref{china244} is to estimate, for a fixed tangent vector $\xi$,  the second order Taylor formula for the variation of KL divergence     $ D_{ K L} (\mu,\mu_\theta) $
with an infinitesimal variation of $\theta$. The vector $\xi$ is tangent at $\mu$ on the infinite dimensional manifold $\mathcal{G}$ (according to \cite{KGLM}).

In other words, we would like to estimate a Thermodynamic Formalism version of  (1.24) in \cite{Ama}.

\eqref{china244} is related to the Fisher Information which is a quite important concept in Statistics, Information Geometry  and Statistical Mechanics (see   section 7.5 in \cite{Roy},  Section 7 in \cite{Cat}, Section 1.6.4 in \cite{Arovas},  \cite{Rupp}, \cite{Sa}, \cite{Ji} or \cite{Espo}. The parameter $\theta$  can be considered as time  in a weakly relaxing setting of non-equilibrium (see  expression (29) in section E in \cite{Altaner}).

\smallskip

It will be necessary first to present some classical definitions and results from Thermodynamic Formalism.

\begin{definition}  Assume $\mu$ is the H\"older Gibbs probability for the potential $J$.  The  {\it asymptotic variance} for the H\"older function
$\xi:\Omega \to \mathbb{R}$ with respect to $\mu$ is
\begin{equation} \label{erto56}  \text{asy-var} (\xi, \mu)=\lim_{n \to \infty} \frac{1}{n} \int (\sum_{i=0}^{n-1}  \xi \circ \sigma^i -n\, \int \xi d \mu)^2 \,d \mu.
\end{equation}
\end{definition}

A C.L.T. can be proved for the function $\xi$ and the probability $\mu$ (see \cite{PP}) and  the variance of the limit Gaussian distribution is $ \text{asy-var} (\xi, \mu).$

If  $\int \xi d \mu=0$ we get
\begin{equation} \label{erto}  \text{asy-var} (\xi, \mu)=\lim_{n \to \infty} \frac{1}{n} \int (\sum_{i=0}^{n-1}  \xi \circ \sigma^i )^2 \,d \mu.
\end{equation}

 For $\xi$ fixed, we denote $\mathfrak{P}(\theta) $ the pressure of the potential $A_\theta=\log J +\, \theta \,\xi$.

It follows from \cite{PP}:

\begin{theorem} \label{gro1} Given the H\"older Jacobian $J$ and  the H\"older function
$\xi:\Omega \to \mathbb{R}$,  then
\begin{equation}  \frac{d^2 \, \mathfrak{P}(\theta) }{d^2 \theta}|_{\theta=0}\,=\text{asy-var} (\xi, \mu),
\end{equation}
where  $\mu$ is the H\"older Gibbs probability for the potential $J$.
\end{theorem}

The value $\text{asy-var} (\xi, \mu)$ describes susceptibility with respect to the variation of  $\xi$.

\begin{definition}
Given  $\xi$, denote $V_n^{\xi}: \Omega \to \mathbb{R}$ the function defined by
\begin{equation}x= (x_1,x_2,...,x_n,..) \,\to \, V_n^{\xi}(x)  =
\frac{d}{d \theta}|_{\theta=0} \log \mu_\theta  (\overline{x_1,x_2,...,x_{n-1}}),
\end{equation}
where $\overline{x_1,x_2,...,x_{n-1}}$ is the corresponding cylinder set of size $n$.
\end{definition}

Following Definition 4.3  in \cite{Ji} we define:
\begin{definition}  Given  $\xi$   we call
$$\mathfrak{F}_{\mu,\xi}^n=\int  V_n^\xi (x)^2 \,d \mu (x)=$$
\begin{equation} \int (\frac{d}{d \theta}|_{\theta=0} \log \mu_{\log J  + \theta \xi}  (\overline{x_1,...,x_{n-1}})\,)^2 d \mu (x)
\end{equation}
the Fisher information at time $n$ for $\xi$ and $\mu$.
\end{definition}

\begin{definition}  Given  $\xi$  we call
\begin{equation} \mathfrak{F}_{\mu,\xi}  = \,\lim_{n \to \infty} \frac{1}{n} \int V_n^\xi (x)^2 \,d \mu (x)
\end{equation}
the Fisher information for the tangent vector $\xi$ and $\mu$.
\end{definition}

A nontrivial result is Proposition 4.4 in \cite{Ji} which claims:

\begin{theorem} \label{gro2} Given the H\"older Jacobian $J$ and the direction $\xi$, then
\begin{equation}  \mathfrak{F}_{\mu,\xi} =\text{asy-var} (\xi, \mu),
\end{equation}
where $\mu$ is the H\"older Gibbs probability for the potential $J$.
\end{theorem}

The above result does not require that $\int \xi d \mu=0,$ but we will  use the claim under such hypothesis.

The results in \cite{Ji} about Fisher information are related to the asymptotic efficiency of maximum likelihood estimators (see section 4 in \cite{Ji}).

The Fisher information also provides a metric structure for a statistical manifold (see \cite{Ama} or the chapter on information geometry in \cite{Cat}).
The metric allows you to compute a distance between neighboring probability distributions which is a measure of the extent to which the two distributions can be statistically distinguished from each other. In chapter 10 in \cite{Cat} the Fisher information  was used in a crucial way to derive Quantum Mechanics from entropy.


\smallskip

In \cite{KGLM} and \cite{LR} it is considered a natural Riemannian metric in the infinite-dimensional manifold of H\"older Gibbs probabilities $\mathcal{G}$. Some points in the manifold $\mathcal{G}$ have positive curvature and others have negative curvature (see \cite{LR}).
We point out that in \cite{LR} explicit expressions for the curvature can be obtained. This Riemannian metric is not compatible with the one associated with the $2$-Wasserstein distance.

Given a probability  $\mu$ (with H\"older Jacobian $J$) the tangent space to $\mathcal{G}$ at $\mu$ is given by the set of H\"older vectors $\xi:\Omega \to \mathbb{R}$ which are on the kernel of the operator $\mathcal{L}_{\log J}.$ In this case we get from \eqref{bact3} that $\int \xi d \mu=0$.

 Moreover the asymptotic variance
 \begin{equation} \label{eai} \text{asy-var} (\xi, \mu)=  \mathfrak{F}_{\mu,\xi} =|\xi|^2,
 \end{equation}
 where $|\xi|^2$ denotes the square of the Riemannian norm of the tangent vector $\xi$ at $\mu$ on $\mathcal{G}$ (see  sections 3 and  4 in \cite{KGLM}). This is the reason why we say that the Riemannian metric we consider on $\mathcal{G}$ is natural.

 In this case, it follows from Theorems \ref{gro1} and \ref{gro2} the relation with the Fisher information
\begin{equation} \label{eai1}
\frac{d^2 \, \mathfrak{P}(\theta) }{d^2 \theta}|_{\theta=0} = |\xi|^2 =  \mathfrak{F}_{\mu,\xi}.
\end{equation}

Note that  $\int \xi d \mu=0$, for a given H\"older function $\xi:\Omega \to \mathbb{R}$, does not mean that $\xi$ is a tangent vector to $\mathcal{G}$ at $\mu$.

Consider two H\"older Jacobians $J_1$ and $J_2$.
Denote $A_\theta =\log J_2 +\, \theta \,\xi$,  where $\xi$ is a tangent vector at $\mu_2$ on $\mathcal{G}$ and  $\theta \in \mathbb{R}$.  The associated H\"older Jacobian  is denoted by $J^\theta$ and $\mu^\theta$ is the associated equilibrium state for  $A_\theta$ (or, for $\log J_\theta$).

In a similar way as in \eqref{pois} we  get that the  Jacobian $J^\theta$  satisfies
\begin{equation} \label{rewo12}\log J^\theta= \log J_2 +  \theta \, \xi +  \log \varphi_\theta - \log \varphi_\theta (\sigma) - \log \lambda_\theta,
\end{equation}
where $\log \lambda_\theta= \mathfrak{P}(\log J_2 +  \theta \, \xi  ).$

It is known (see \cite{PP}) that for a continuous function $w:\Omega \to \mathbb{R}$ (not necessarily satisfying $\int w d \mu_2=0$)
\begin{equation} \label{rewo}   \frac{d}{d \theta}|_{\theta=0}   \mathfrak{P}(\log J_2 +  \theta \, w  )  =\int w d \mu_2.
\end{equation}

Question: For   fixed H\"older Gibbs probability $\mu_1$ and Jacobian $J_2$, estimate on the direction $\xi: \Omega \to \mathbb{R}$ (not necessarily tangent at $\mu_2$) on the base point $\mu_2 \in \mathcal{G}$, the first derivative
$$ \frac{d}{d \theta}|_{\theta=0}   D_{K L}      (\mu_1,\mu^\theta)  =\frac{d}{d t}|_{\theta=0}  (   - \int \log J^\theta d \mu_1 + \int  \log J_1 d \mu_1 ). $$

We will address this question.

This  estimate  is the Thermodynamic Formalism version of  (1.24) in \cite{Ama}.

Note that from \eqref{rewo}, \eqref{rewo12}  and the invariance of $\mu_1$
$$ \frac{d}{d \theta}|_{\theta=0}   D_{K L}      (\mu_1,\mu^\theta)  = $$
$$ \frac{d}{d \theta}|_{\theta=0} \,[ - \int (\log J_2 +  \theta \, \xi +  \log \varphi_\theta - \log \varphi_\theta (\sigma) - \log \lambda_\theta ) d \mu_1 +  \int  \log J_1 d \mu_1      ]=$$
$$ \frac{d}{d \theta}|_{\theta=0} \,[ - \int (\log J_2 +  \theta \, \xi - \log \lambda_\theta ) d \mu_1 +  \int  \log J_1 d \mu_1      ]=$$
\begin{equation} \label{china11}  - \int \xi d \mu_1 + \int \xi d \mu_2 .
\end{equation}

In the case we assume that $\xi$ is tangent at $\mu_2$ (which implies $\int \xi d \mu_2 =0$), then we get
\begin{equation} \label{china0} \frac{d}{d \theta}|_{\theta=0}   D_{K L}      (\mu_1,\mu^\theta)  =   - \int \xi d \mu_1 + \int \xi d \mu_2 =- \int \xi d \mu_1.
\end{equation}

Section E in \cite{Altaner} call a non equilibrium of {\it   strongly relaxing} if \eqref{china0} is less or equal zero  (see expression(31) in \cite{Altaner}). This seems to be not  always the case here.

\begin{proposition} \label{pppr} Assume that $\xi$ is a tangent vector to $\mathcal{G}$ at $\mu_2$, then,
\begin{equation} \label{china88}  D_{K L}      (\mu_1,\mu_ 2 + d \xi) =  -  \int \xi d \mu_1 + \frac{1}{2} \int \xi^2 d \mu_2  +
o (|d \xi|^2) ,
\end{equation}
where $ \int \xi^2 d \mu_2$ is the Fisher  information.
\end{proposition}

\begin{proof}
Now we consider the second derivative. From \eqref{eai} and \eqref{eai1} we get
$$ \frac{d^2}{d^2 \theta}|_{\theta=0}   D_{K L}      (\mu_1,\mu^\theta)  =\frac{d^2}{d^2 \theta}  (   - \int \log J^\theta d \mu_1 + \int  \log J_1 d \mu_1 )= $$
\begin{equation} \label{china15} \frac{d^2}{d^2 \theta}|_{\theta=0} -\,[ \int (\log J_2 +  \theta \, \xi - \log \lambda_\theta ) d \mu_1 ]=  \int \xi^2 d \mu_2.
\end{equation}

The claim follows  from \eqref{china0} and \eqref{china15}.

\end{proof}

Theorem \ref{corre27} will be  a consequence of the Proposition \ref{corre} which follows at once  from Proposition \ref{pppr}.

\begin{proposition} \label{corre} Assume $\mu_1=\mu_2$ and  $\xi$ is a tangent vector to $\mathcal{G}$ at $\mu_1$, then,
\begin{equation} \label{china}  D_{K L}      (\mu_1,\mu_ 1 + d \xi) =   \frac{1}{2} \int \xi^2 d \mu_1  +
o (|d \xi|^2) ,
\end{equation}
where $ \int \xi^2 d \mu_1$ is the Fisher  information.
\end{proposition}

The above second-order Taylor formula is the analogous of expression (1.24) in \cite{Ama}.












\section{Thermodynamic of gases} \label{gaz}

In this section, we briefly explain for the mathematician some basic concepts of gas Thermodynamics which will be our main focus in the next sections.

The reader familiar with the topic of this section can skip it.

In gas Thermodynamics  the equilibrium state is characterized by internal energy $U$, volume $V$, temperature $T$, etc.
Moreover, only differences of energy $\Delta U$ (or, differences of  heat $\Delta Q)$, rather than absolute values of energy $U$ (or, heat $Q$), have physical significance (page 12 in \cite{Callen}).  This point of view will be followed in the other sections, for instance, in expressions \eqref{p275} and \eqref{p276} in section \ref{gazes} -  which considers  the Thermodynamic Formalism setting.

Spontaneous heat transfer from hot to cold is an irreversible process.
The second law of thermodynamics states: heat cannot spontaneously flow from a colder location to a hotter location. The concept of irreversibility is linked to the concept  of entropy production (see \cite{MN}, \cite{MNS}, \cite{Cat}, \cite{Wang}, \cite{ThoQi}, and \cite{LM}) which will be considered in sections \ref{fish} and  \ref{EP}.

A {\it thermodynamic process} (also called {\it thermodynamic operation}) may be defined as the energetic evolution of a thermodynamic system proceeding from an initial state to a final state. We will elaborate on that.
In our Thermodynamic  setting the time is not a relevant variable - the {\it quasi static-regime} - and  this means that  the thermodynamic  processes we consider are such that  the changes are slow enough for the system to remain in internal equilibrium.
The terminology  "thermostatics" would  actually be more appropriate   than "thermodynamics"    because we will not consider here  the change of the  physical system with time. The quasi static-regime describes  a type of  non equilibrium  where a certain kind of equilibrium still happens.

 Section \ref{fish},  where time plays some role, is an exception here.

In the quasi-static regime
the {\it work} $W$ and {\it volume} $V$ are related by the  {\it pressure} $p$ via the equation
\begin{equation} \label{t1} d W =-  p\, d V.
\end{equation}

A version of this expression will be  described by  \eqref{p575} in Section \ref{gazes}.

The meaning of expression \eqref{t1} can be observed in the case of a gas enclosed in a cylinder with a movable piston  (see figure \ref{fig:Graf}). If the volume of the system is decreased slowly in a continuous way,  work is done in the system, increasing its energy. If the variation of work is positive this will increase the energy of the system - the gas realizes work on the external medium.  If the variation of work is negative we say that  the external medium realizes work on the gas. These thermodynamical quantities are controlled by an experimenter.

For a more general form of \eqref{t1} see   \eqref{poup}.

We denote by $Q$ the {\it heat}. In an infinitesimal quasi-static processes, the variation of the quasi-static heat $d Q$ satisfies the equation
\begin{equation} \label{t2}  d Q = d U - d W.
\end{equation}

The above expression is a common form of the {\bf First Law of Thermodynamics} and  describes a form of conservation of energy (see Chapter 2 in  \cite{Arovas}). The variation of heat is related to the variation of energy and the variation of work.

A version of expression \eqref{t2} will be  described by expression \eqref{p277}     in Section \ref{gazes}.


\smallskip

We point out that the conventional signal minus we used in \eqref{t2} (before the work $W$) can be avoided depending if we consider the  work on the system or on the external medium. That is, the form $d Q = d U + d W$ also appears in the literature.

Note that  we also get from \eqref{t2} and \eqref{t1}
\begin{equation} \label{t3}  d Q = d U + p\, d V.
\end{equation}

It is implicit in the above notation that given a certain state the value $U$ is the energy of the state and $Q$ its heat. $\Delta Q$ describes the variation of heat under the action of the thermodynamic operation. Same thing for the variation $\Delta U$.

The {\it Entropy} $S$ (denoted in such way in this section) is a concept for states which are in gas-thermodynamical  equilibrium. A fundamental postulate is that under a certain set of constraints, the equilibrium of the system is achieved for the state which maximizes entropy, among states satisfying these constraints (see the MaxEnt method in section \ref{gazes}). This claim corresponds in Information Theory to the problem of optimizing capacity cost for a hard channel; that is,  maximizing entropy among probabilities that accomplish a certain fixed mean cost (see \cite{CL}).

 \begin{figure}[h] \label{fig:Graf}
\center
\includegraphics[height=4cm,width=12cm, trim=0.1in 0.08in 0.1in 0.08in, clip]{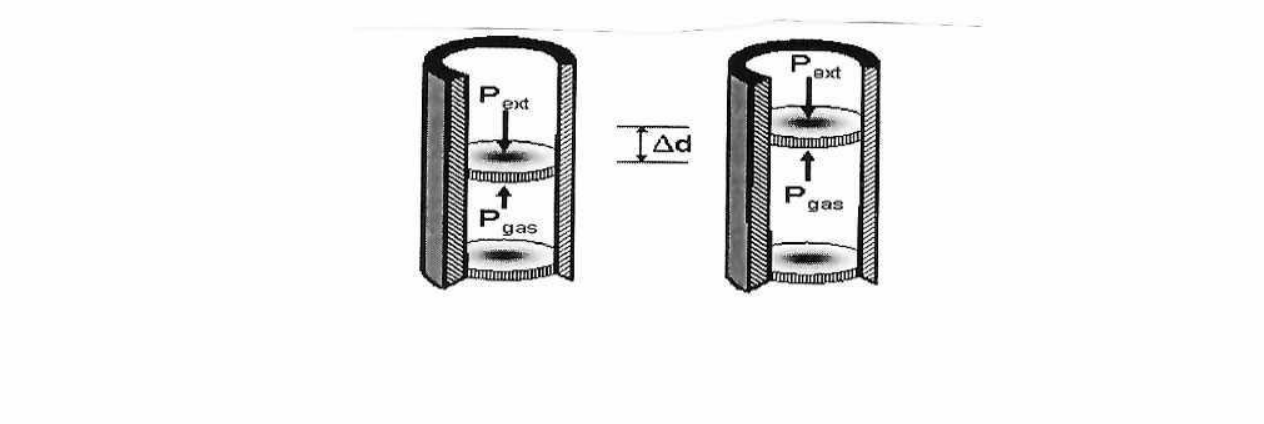}
\caption{Work, volume and pressure.   The image describes a slow and continuous variation in the volume of the piston chamber. The expression $d W = - p\, d V$ is true in the quasi-static regime. }
\label{fig:Graf}
\end{figure}

The variation of total entropy $d_t S$ can be written as a sum of two terms
\begin{equation} \label{entropia} d_ t S= T\, d_e  S+ d_i S,
\end{equation}
where $d_e S$ is the entropy supplied to the systems by its {\it surroundings} and $d_i S$ is the entropy produced {\it inside} the system (see (1)  page 20 in \cite{Groot}).


The variation of entropy $S$ and the variation of heat  $Q$ are related by the temperature:
\begin{equation} \label{t5}  d Q = T \, d_e S.
\end{equation}

The above is true for the so-called closed systems, which may only exchange heat with their surroundings (see (4) page 20 in \cite{Groot}). 

A version of this relation in our setting appears in \eqref{rert}   in Section \ref{gazes}.

A Thermodynamical Systems is called
{\it reversible}, if $ d Q = T dS$,  {\it spontaneous} if   $ d Q < T dS$, and
{\it adiabatic} if $d Q =0$ (see section 2.2.1 in \cite{Arovas}).

  The {\bf Second Law of Thermodynamics} claims that $d_i S$ must be zero for reversible transformations and positive for irreversible transformations of the system. The entropy $d_e S$ can be positive, zero, or negative. For adiabatic insulated systems $d_e S=0.$

Our hypothesis on Theorem \ref{yyt7}, about conditions for the increase of entropy (Second Law),  should then correspond to the increase of $d_i S$.

From the above, \eqref{t2}, \eqref{t5}  and \eqref{t3} we get the quasi-static equations
\begin{equation} \label{t6}
T\, d S = d U - d W=  d U + p\, d V.
\end{equation}

 $U, S, V $  are known as {\it extensive} variables and the $T, p$  as  {\it intensive} variables.



The equation
\begin{equation} \label{t17} \frac{d S}{d U} =  \frac{1}{T} \,,\end{equation} is known as {\it the fundamental Gibbs equation} (see (5.37) in \cite{Cat}).  A version of this equation is presented in \eqref{rert} in Section \ref{gazes}.

The {\bf Second Law of Thermodynamics} was formulated in a
strong form by Gibbs (1878): {\it for an irreversible processes not only does the
entropy tends to increase, but it does increase to the maximum value allowed by
the constraints imposed on the system}.

The MaxEnt method of Jaynes is a natural way to describe the above statement (see \cite{Jaynes} and Section V.b in \cite{Shore}). In the Thermodynamic Formalism setting this claim is described in Section \ref{gazes} (see expression \eqref{p4}).

The {\it Helmholtz free energy} $F$ (see section 2.8.2  in \cite{Arovas}) of a state (see also expression \eqref{p2} in Section \ref{gazes})  is given by
\begin{equation} \label{t8}
F = U - T \, S,\,\,\text{where}\,\,T \,\text{and}\,\,V \,\,\text{do not depend of the state}.
\end{equation}

It is implicit in the above notation that given a certain state the value $U$ is the energy of the state and $S$ 
.

An equilibrium state is a state which minimizes the Helmholtz free energy  (see Section \ref{gazes}). This variational formulation is analogous to the variational principle of minimum action of Classical Mechanics.

The problem of finding the probability minimizing Helmholtz free energy is equivalent (in Thermodynamic Formalism) to finding the probability which maximizes topological pressure (see \eqref{bact7}  and   \eqref{p2}).

The
{\it Gibbs free energy  $G$} is described by
\begin{equation} \label{t81} G = U + p\, V  - T \,S,
\end{equation}
where $p$ is pressure and $V$ is volume. States minimizing the Gibbs free energy extend the scope of the meaning of equilibrium states when one adds  pressure and volume to the problem.
In  \eqref{p575} in section \ref{gazes} we describe the concept of  pressure in Thermodynamic Formalism. In this way, one can consider in this theory (we will not address this issue here) a  broader class of problems related to the introduction of the term $- p\, V$ on the maximizing problem (related to minimizing \eqref{t81}).

\begin{remark} \label{uut} In \cite{Cat} the author describes in Section 5.7 a version of the Second Law of Thermodynamics  for an irreversible system which in simplified terms is the following: consider a time-dependent Hamiltonian $H(t)$ and an initial equilibrium state $f^{can}$. The initial condition evolves according to Liouville equation and after time $t'$  attains the state $f(t')$, which is not in equilibrium. Now, in some way, using the MaxEnt method one can get from $f(t')$ an associated equilibrium state $f^{can} (t')$. Finally, it follows from the  computations on the text that the entropy of $f^{can} (t')$ is larger than the initial entropy of $f^{can}$. Theorem \ref{yyt721} in our Section \ref{pre} describes a similar kind of behavior under the action of the dual of the Ruelle operator.
\end{remark}




For a mathematical formulation   of thermodynamics via contact geometry we refer the reader to \cite{Schaft} and \cite{Balian}.

We will not elaborate much on the concept of  thermodynamic operation on the setting of thermodynamic of gases. We just mention that the quasi-static action of a piston on figure \ref{fig:Graf} describes a certain type of thermodynamic operation (as a function of the change of volume in a continuous way, for instance). We
consider a natural form of thermodynamic operation in sections \ref{pre}, \ref{gazes} and  \ref{TO}  (the Thermodynamic Formalism setting)





The Second law of Thermodynamics considers an isothermal process (temperature is fixed). The system is in contact with a  large single heat bath at temperature $$T =\frac{1}{\beta \, \times \text{Boltzmann constant} },$$
where $\beta>0$.  The letter $\beta$ is in according to a
common tradition in Statistical Mechanics.

\section{Thermodynamic Formalism and Thermodynamic of gases} \label{gazes}



We would like to interpret the concepts described in  section \ref{gaz} (like variation of work, heat, internal energy, etc.,) within a vision of classical Thermodynamic Formalism
and Shannon-Kolmogorov entropy on the Bernoulli space $\Omega=\{1,2,...,d\}^\mathbb{N}$.  We will consider in the end of this section a  continuous  variation of Thermodynamic quantities.

One of the  main concepts in Thermodynamic Formalism is the  Topological  Pressure (see Definition \ref{fufu}).

Consider a  H\"older  continuous function $ M: \Omega \to \mathbb{R}$.

The {\it Topological Pressure} of the potential $\frac{-\, M}{T}$, where $T $ is temperature, is the value
\begin{equation} \label{p1}
\mathfrak{P} (-  \frac{M}{T}) = \sup \{ h(\mu) - \frac{1}{T}\int   M d \mu\,|\, \mu \, \text{invariant for the shift}\, \sigma\},
\end{equation}
where $M$ corresponds to the {\it extensities}, $h(\mu) $ is the {\it  Shannon-Kolmogorov entropy} of $\mu$ and $T$ is {\it temperature}.   A particular case of interest in Statistical Mechanics is when $M=H$, where $H$  is the Hamiltonian, but can be also something more general. We deliberately used the notation $\mathfrak{P} $ for Topological Pressure in order not to confuse  the concept just defined above with the concept of pressure $p$ as described in   Section \ref{gaz}.

An {\it equilibrium state for  $\frac{-\, M}{T}$} is  a shift invariant probability on $\Omega$ attaining the maximal value $\mathfrak{P} (\frac{-\, M}{T}).$

When $M$ is H\"older the Ruelle operator is a quite useful tool for understanding the more important properties of the associated equilibrium state. For instance, the main eigenvalue $\lambda$ of the Ruelle operator $\mathcal{L}_{\frac{-\, M}{T}} $ satisfies
$\log \lambda = \mathfrak{P} (\frac{-\, M}{T}).$ The main eigenvalue $\lambda$ and the main eigenfunction $\varphi$ of the Ruelle operator  are analytic functions of $M$ (see \cite{PP}).

It is usual in the literature to denote by $\beta$ the value $1/T$. The Hamiltonian $H$ describes energy and therefore
in the context of Physics we are interested in invariant probabilities maximizing
\begin{equation} \label{p73}
  \mathfrak{P}(- \, \beta H)=\sup \{ h(\mu) - \beta \int   H d \mu\,|\, \mu \, \text{invariant for the shift}\, \sigma\}.
\end{equation}

The minus sign that goes before $H$ is natural and compatible with the  measurements  in  laboratory showing that, for the equilibrium probability maximizing \eqref{p73}, the sets with strings in $\Omega$ with high value of energy have smaller probability.


A typical example of function $M$ is
\begin{equation} \label{vvv}
M= H +  p_1\, M_1 + p_2 M_2+ ...+ p_j\, M_j
\end{equation} (that is $H$ is one of the elements in the sum). The functions $M_1,M_2,...,M_j$, are the so called  working functions (or, extensities) and the $p_1,p_2,...,p_j$ are called the {\it intensities}. The formula for the variation of work is now
\begin{equation} \label{poup}  d W = -  \sum_{l=1}^j p_l \, d M_l.
\end{equation}

Volume $V$ could be one of the  working functions $M_l$ (see \eqref{t1}) and in this case the corresponding  intensity $p_l$ would describe pressure.

When we compare the topological  pressure $\mathfrak{P} (H) $  with Helmholtz free energy
\begin{equation} \label{p24 }F(H) = H - T \, S
\end{equation}  (which corresponds to expression \eqref{t8}), we get the relation
\begin{equation} \label{p2}
\mathfrak{P} (-H) = - \frac{F(H)}{T}.
\end{equation}

For a fixed temperature $T$ to maximize pressure $\mathfrak{P}$ is equivalent to minimize $F$. Equilibrium states are the also the  ones which  minimize Helmholtz free energy. When $T=1$ we get
\begin{equation} \label{p299} -\,\mathfrak{P} (-H) =  F(H) = \text{ Helmholtz free energy for\, } H.
\end{equation}

The case when $M = H + p \,  V$ in \eqref{vvv} can be considered as  related to the analysis of states minimizing Gibbs free energy (see \eqref{t81}).
\smallskip

As a working example, we address the case where \eqref{vvv} depends on three variables $(z_0,z_1,z_2)$.
Consider the family of  H\"older functions $f_j^v :\Omega \to \mathbb{R}$, $j=0,1,2$, $v \in \mathbb{R}$, and  the variable $z=(z_0,z_1,z_2)\in \mathbb{R}^3$.

We take above $- \frac{1}{T}\,M=z_0 \,f_0^v  + z_1 \, f_1^v + z_2 \,f_2^v$.
The {\it extensive functions}  $f_0, f_1,f_2$,  could represent in physical problems quantities as
particle numbers, magnetic moments,  electrical charge, etc. They depend on an external parameter $v\in \mathbb{R}$. All of the above could also be considered for $z=(z_0,z_1,z_2,...,z_n)\in \mathbb{R}^n$, but we want to simplify the notation.

For each fixed $v$ take
$$z = (z_0,z_1,z_2) \to \mathfrak{P}^v (z_0,z_1,z_2) =\mathfrak{P} (\, z_0 \,f_0^v  + z_1 \, f_1^v + z_2 \,f_2^v\,)= $$
\begin{equation} \label{p3}   \sup \{ h(\mu) +  \int    (\,z_0 \,f_0^v + z_1 \, f_1^v + z_2 \,f_2^v)\,d \mu\,\,\,|\,\,\, \mu \, \text{invariant for the shift}\, \sigma\}.
\end{equation}

 The value $\int    (\,z_0 \,f_0^v + z_1 \, f_1^v + z_2 \,f_2^v)\,d \mu$ will be called the {\it internal energy} $U^v$ of the state $\mu$ for the parameter $v$.

We denote $\mu_{z_0,z_1,z_2}^v$ the invariant probability  which maximizes the unconstrained problem \eqref{p3}.
In the case there exists a natural probability $d v$ on the set of parameters $v$ we can be interested in the probability $\int \mu_{z_0,z_1,z_2}^v\, d v.$
The parameter  value $v$ (called macroscopic control parameter in \cite{Altaner}) can correspond to volume or externally applied magnetic field (see \cite{Cat}).

We say that a function $L:\Omega \to \mathbb{R}$ is {\it cohomologous} to a constant if there exist $\varphi:\Omega \to \mathbb{R}$
and $c$ such that
\begin{equation} \label{bact8}  L = \varphi \circ \sigma - \varphi + c.
\end{equation}

For each value $v$, we will assume Hypothesis A of \cite{Lall1}: $f_j^v$, $j=0,1,2$, satisfies the property, if $(a_0,a_1,a_2)$ is such that
$$ \,a_0 \,f_0^v + a_1 \, f_1^v + a_2 \,f_2^v$$
is cohomologous to  constant, then $a_j=0$, $j=0,1,2$.

Now, for each fixed $v$ and the variable $x=(x_0,x_1,x_2)$ take
$$  (x_0,x_1,x_2) \to \alpha^v(x_0,x_1,x_2) = \sup\{ h(\mu)\,|\, \int  f_0^v d \mu= x_0, \int  f_1^v d \mu= x_1,$$
\begin{equation} \label{p4}  \int f_2^v d \mu=x_2,\, \, \text{where}\,\,\mu \, \text{is invariant for}\, \sigma \}.
\end{equation}

We denote $\tilde{\mu}_{x_0,x_1,x_2}^v$  the   invariant probability such that its entropy $ h (\tilde{\mu}_{x_0,x_1,x_2}^v)$ maximizes  the constrained problem  \eqref{p4}, that is the arg max  of \eqref{p4}. This probability, 
obtained via an inference procedure, 
should be called the {\it a posteriori} probability for the fixed constraints functions  $f_0^v,f_1^v,f_2^v$ and values $v_0,v_1,v_2$. The existence of 
$\tilde{\mu}_{x_0,x_1,x_2}^v$ follows from the semicontinuity of the entropy (see also \cite{Lall1}).

The reasoning behind finding $\tilde{\mu}_{x_0,x_1,x_2}^v$ is
called the {\it method of maximum entropy} and describes the original point of view of R. Clausius for the Thermodynamics of gases. We want to relate \eqref{p4}
with \eqref{p3} .

An increase in entropy  means an increase  in the uncertainty of the information.
The   invariant probability  $\tilde{\mu}_{x_0,x_1,x_2}^v$ such that its entropy maximizes  the constrained problem  \eqref{p4} can be understood as the probability   optimizing the  capacity costfunction of a hard constrained channel (see page 1168 in \cite{CL} and also \cite{Ben}).  This line of reasoning is naturally justified  by the so-called Jaynes principle:

 {\it The unbiased guess of Jaynes  is the method of seeking among all
possible distributions the one  which comprises maximum entropy. Any
other would comprise unjustified prejudices.}
\smallskip

Consider for fixed $v$ the law
$$x=(x_0,x_1,x_2) \to\gamma^v(x_0,x_1,x_2)  =$$
\begin{equation} \label{p5}= \sup_{(z_0,z_1,z_2) \in \mathbb{R}^3}\{ x_0 \, z_0 + x_1\,z_1 + x_2 \, z_2 -\, \mathfrak{P}^v (z_0,z_1,z_2)  \} ,
\end{equation}
which is the Legendre Transform of $\mathfrak{P}^v  (z_0,z_1,z_2) .$

For each $v$, given $x$ there exists a unique $z^v=z^v (x)$ such that
\begin{equation} \label{p27} x= (x_0,x_1,x_2) \to z^v =(z_0^v,z_1^v,z_2^v) = (z_0^v(x),z_1^v(x),z_2^v(x))   \in \mathbb{R}^3,
\end{equation}
where $z^v =(z_0^v,z_1^v,z_2^v)$ realizes the supremum in \eqref{p5}.
The existence of  $z^v$ follows from the convexity of $\mathfrak{P}^v.$

From (d)  page 161 in \cite{Lall1} we get that for fixed $v$, given $x$ there exists $\tilde{z}^v= (\tilde{z}_0^v,\tilde{z}_1^v ,\tilde{z}_2^v)$ such that
$$\nabla \mathfrak{P}^v (\tilde{z}_0^v,\tilde{z}_1^v,\tilde{z}_2^v) = (x_0,x_1,x_2).$$

It is known that for  each fixed $v$ we get $\alpha^v(x) = - \gamma^v(x)$ (see \cite{CL} or \cite{Lall1}).

We say that $z^v = (z_0^v,z_1^v,z_2^v)$  is the dual pair of $x=(x_0,x_1,x_2)$ if $z=(z_0^v,z_1^v,z_2^v)$ realizes the supremum in \eqref{p5}. This means that
$$\nabla \mathfrak{P}^v (z_0,z_1,z_2) =$$
\begin{equation} \label{bact15}  (\frac{\partial \mathfrak{P}^v }{\partial z_0} (z_0,z_1,z_2) ,\frac{\partial \mathfrak{P}^v }{\partial z_1}  (z_0,z_1,z_2),\frac{\partial \mathfrak{P}^v }{\partial z_2} (z_0,z_1,z_2))= (x_0^v,x_1^v,x_2^v),
\end{equation}  which is equivalent to
$\nabla \alpha^v (x_0,x_1,x_2) = (z_0,z_1,z_2)$ (see Lemma 1 page 1170 in \cite{CL} or (h) page 162 in \cite{Lall1}).

The equality $\nabla \alpha^v (x_0,x_1,x_2) = (z_0,z_1,z_2)$ corresponds to expression (4.77) in \cite{Cat}.

From the above we get, for fixed $v$,   the bijective relation
\begin{equation} \label{bact16}  (x_0^v,x_1^v,x_2^v) \Leftrightarrow (z_0^v,z_1^v,z_2^v).
\end{equation}

We get the equality (see (i) page 162 in \cite{Lall1})
\begin{equation} \label{p279}\alpha^v (\tilde{x}_0^v,\tilde{x}_1^v,\tilde{x}_2^v) =  \mathfrak{P}^v (\tilde{z}_0^v,\tilde{z}_1^v,\tilde{z}_2^v) -    x_0^v  z_0^v + x_1^v z_1^v+ z_2^2 x_2^v.
 \end{equation}

 If $z^v=(z_0,z_1,z_2)$  is the dual pair of $x^v=(x_0,x_1,x_2)$, then,  $\tilde{\mu}_{x_0,x_1,x_2}= \mu_{z_0,z_1,z_2}$. This shows that the  method of maximum entropy  and the principle of maximizing pressure coincide.
 \smallskip

The matrix
\begin{equation} \label{bact1} \mathfrak{SP}= \left(\frac{\partial^2 \mathfrak{P}^v }{\partial z_i  \partial z_i}  (z_0,z_1,z_2) \right) _{i,j=1,2,3}
\end{equation}
is called the {\it susceptibility pressure matrix} (see (2.2) page 33 in  \cite{Sch} or Section 4.5.8  in \cite{Arovas}).

For dual pairs one gets (see Lemma 3 in \cite{LM} page 1173 or  (h) page 162 in \cite{Lall1})
\begin{equation} \label{bact2}  \left(\frac{\partial^2 \mathfrak{P}^v }{\partial z_i  \partial z_i}  (z_0,z_1,z_2) \right)_{i,j=1,2,3}^{-1}=\,-\ \left(\frac{\partial^2 \alpha^v }{\partial x_r  \partial x_s}  (x_0,x_1,x_2) \right)_{r,s=1,2,3}=\mathfrak{SE}.
\end{equation}

The  {\it susceptibility entropy matrix} $\mathfrak{SE}$ is minus the inverse of the {\it susceptibility pressure matrix} $\mathfrak{SP}$ .
The    matrix $\mathfrak{SE}$ is also called  the {\it fluctuation matrix} (see (2.5) in page 35 in \cite{Attard}). It is negative definite because the entropy of probabilities $\mu$ nearby $\tilde{\mu}_{x_0,x_1,x_2}^v$ (the probability realizing the supremum of $\alpha^v$ in \eqref{p4}) are smaller than $h(\tilde{\mu}_{x_0,x_1,x_2}^v).$

\smallskip

For a fixed  continuous function $f:\Omega \to \mathbb{R} $ the function
 $$ v \to \int f\, \tilde{\mu}_{x_0,x_1,x_2}$$
 is analytic on $v$ (see  \cite{BCV} for explicit formulas).

 We consider now that the variable $v$ describes volume and we would like to understand the variation of other important thermodynamic quantities with $v$.

 For fixed $f$, $x_j$, $j=0,1,2$, we denote
 \begin{equation} \label{p274} \int f\,\frac{d}{d v}|_{v =v_0}\, \tilde{\mu}_{x_0,x_1,x_2}^v:= \frac{d}{d v}|_{v =v_0}\int f\, \tilde{\mu}_{x_0,x_1,x_2}^v .
 \end{equation}

 For $x=(x_0,x_1,x_2)$ fixed, we denote
 \begin{equation} \label{p276}\mathbb{F}_k (v) = \int f_k^v \,\,d \tilde{\mu}_{x_0,x_1,x_2}^v ,
 \end{equation}
 $k=0,1,2$, the $k$ {\it internal energy}.

Then, for fixed $k=0,1,2$, $x_j$, $j=0,1,2$, we get  at the point $v_0$
\begin{equation} \label{p273}
\frac{\delta \mathbb{F}_k}{\delta v}= \frac{ d \mathbb{F}_k(v)}{d v}|_{v =v_0} = \int  \frac{d f_k^v}{ d v}|_{v =v_0} d
\tilde{\mu}_{x_0,x_1,x_2}^{v_0} + \int   f_k^{v_0}     \frac{d}{dv}|_{v =v_0}\, \tilde{\mu}_{x_0,x_1,x_2}^v.
\end{equation}

The above expression corresponds to expression (4.78) in \cite{Cat}.

For a fixed $v_0$ and $k$ we are considering an infinitesimal change of the parameter $v$ around $v_0$, which is described by
$ \frac{ d \mathbb{F}_k(v)}{d v}|_{v =v_0}$. If  $\mathbb{F}_k$ describes internal energy and $v$ represents volume, then the first term of the right hand side of \eqref{p273} is
\begin{equation} \label{p275} \frac{\delta W}{\delta v}|_{v =v_0}= \int  \frac{d f_k^v}{ d v}|_{v =v_0} d \tilde{\mu}_{x_0,x_1,x_2}^{v_0}
\end{equation}
which represents the {\it infinitesimal change of work} at $v_0$.

On the other hand, if   $\mathbb{F}_k$ describes internal energy and $v$ represents volume, then the second term of the right hand side of \eqref{p273} is
 \begin{equation} \label{p276} \frac{\delta Q}{\delta v}|_{v =v_0}= \int   f_k^{v_0} \,\,    \frac{d}{dv}|_{v =v_0}\, \tilde{\mu}_{x_0,x_1,x_2}^v.
\end{equation}
which represents the {\it  infinitesimal change of heat} at $v$.

\begin{proposition} Conservation of Energy:
\begin{equation} \label{p277}         \frac{\delta U}{\delta v} =  \frac{\delta W}{\delta v} +   \frac{\delta Q}{\delta v},
\end{equation}
where $\delta U$ represents the {\it  infinitesimal change of internal energy $U$}. This is an infinitesimal form of the {\it First Law of Thermodynamics} (see expression (4.82) in \cite{Cat}).
\end{proposition}
\smallskip

The claim follows from \eqref{p276},  \eqref{p275} and \eqref{p273}.

Consider now $k=1$, that is $- \frac{1}{T} H= \beta\, f_1^v$, $\beta \in \mathbb{R}$,  and the family of potentials $ \beta f_1^v$. The value $\beta$ is equal to
$\frac{1}{T}$, where $T$ is temperature. Denote by $\mu_\beta$ the equilibrium
probability for the potential $ \beta f_1^v$ and $h(\beta)$ its entropy. By definition (see expression (2.5) in \cite{Sch}), using expression \eqref{p73}, the energy for the parameter $\beta$ is given by
 \begin{equation} \label{p297}   E(\beta)= - \int f_1^v d \mu_\beta=- \frac{d P(\beta f_1^v )}{d \beta}= - \frac{d P(-\beta H )}{d \beta}.
 \end{equation}

 For the proof of the above relation see \cite{PP}.

 The second derivative $ \frac{d^2 P(\beta f_1^v )}{d^2 \beta}$ is the asymptotic variance in thermodynamic formalism (see \cite{PP} or \cite{LR}), which is also
 called  the {\it susceptibility} in the thermodynamics of gases (see the first expression in (2.19) in \cite{Sch}).

 We consider now that $E$ is an independent variable. For each value $E$ (in a certain interval range) we can associate
a value $\beta=\beta(E)$ such that
\begin{equation} \label{p497}-\frac{d P(-\beta \,H)}{d \beta}=E
\end{equation}  (note that because we assume Hypothesis A  the topological pressure is a strictly convex analytic function of $\beta$).

Now we write the entropy $h(E) = h( \beta(E))$ as a function of $E$. We want to estimate $\frac{d h (E)}{d E}.$

Note that
$$h(E) =  P(-  \beta(E) H ) + \beta(E) \int H d \mu_{\beta(E)}=  P(-\beta(E) f_1^v ) + \beta(E)\, E.$$

\begin{proposition}
 \begin{equation} \label{rert}\frac{d h(E)}{d E}=  \frac{1}{ T(E)}.
 \end{equation}
\end{proposition}
\begin{proof}

Indeed,
$$ \frac{d h(E)}{d E}= \frac{d P(-\beta(E) f_1^v )}{ d \beta}  \frac{ d \beta (E)}{d E} + \frac{ d \beta (E)}{d E} E + \beta(E)=$$
\begin{equation} \label{rertio} - E  \frac{ d \beta (E)}{d E} +  \frac{ d \beta (E)}{d E} E +\beta(E) =  \beta(E) =  \frac{1}{ T(E)}.
\end{equation}
\end{proof}

The above expression is the classical relation among entropy and time (see expression  (5.37) in \cite{Cat} and section B in \cite{Altaner}). It is called the {\it Gibbs fundamental equation} (see (2.7) and  (2.8) in \cite{Sch}).
We point out that in spin systems the temperature can be negative.

Suppose now that the temperature $T $ is fixed (that is $\beta$ fixed). Then,
\begin{equation} \label{bact17} P(\beta f_1^v) = h (\mu_v) + \beta \int f_1^v d \mu^v_{\beta},
\end{equation}
where $\mu^v_{\beta}$ is the equilibrium probability for the potential $\beta  f_1^v$, and $\beta=1/T$ is fixed.

The value \begin{equation} \label{p575} p =   -  \frac{\delta W}{ \delta v}|_{v =v_0}
\end{equation}
represents  {\it pressure} (see expression (5.42) in \cite{Cat}) when the volume is $v_0$ and  $\frac{\delta W}{ \delta v}$ is given by \eqref{p275}.


\section{Thermodynamic operation in Thermodynamic Formalism} \label{TO}

We are interested in a certain special form of thermodynamic operation and the associated  meaning for nonequilibrium in Thermodynamic Formalism.
We will consider two probabilities $\mu_1$ and $\mu_2$ which  can interact and a certain form of conservation of energy.
We are interested in the variation of heat $\Delta Q$, the variation of work $\Delta W$, etc. More precisely on the expression
$$ \Delta W + \Delta Q = \Delta U,$$
where $\Delta U$  denotes the change in the internal energy.


Consider a normalized H\"older potential $A= \log J_1$ and another one
$B = \log J_2$ and the probabilities $\mu_j$, $j=1,2$, which are H\"older Gibbs respectively for the potentials
$\log J_j$, $j=1,2$.

 We perform the discrete-time thermodynamic operation  
 $$\mu_2 \to  \mathcal{L}^*_{\log J_1}(\mu_2)=\mu_3$$ 
 on the system in equilibrium and we are interested in the variation of work $\Delta W$ and the variation of the free energy of the
 system after the thermodynamic operation. The hypothesis above (where $A= \log J_1$ and $ B=\log J_2$ are normalized H\"older potentials) means that we assume that the temperature is such that to $\beta=1$.

In this section, we consider a discrete-time variation as oppose to the continuous-time variation of volume of  Section \ref{gazes}.

 Our purpose is to adapt the reasoning of section 4 in \cite{Sa} to the Thermodynamic Formalism setting.

Assume that $\mathcal{L}_{\log J_2}^* (\mu_2)= \mu_2,$ where $J_2$ is a H\"older Jacobian.
As we said before the action $\mu_1 \to  \mathcal{L}_{\log J_2}^*(\mu_1)$
should be understood as a   thermodynamic operation on the compose system $\mu_1$, $\mu_2$  which was initially in equilibrium.

\smallskip

Given  the probability $\mu$  the {\it energy} $E(\mu)$  for  the joint system $\mu_1$ and $\mu_2$   is
\begin{equation} \label{energ}
\mu \to E(\mu) =\int \log J_1 d \mu - \int \log J_2 d \mu.
\end{equation}

Given the joint system $\mu_1$ and $\mu_2$,
the expression
\begin{equation} \label{troto}   \Delta Q  (\mu_1,\mu_2)=  \int \log J_2 d \mathcal{L}_{\log J_2}^*(\mu_1) - \int \log J_2 d \mu_1
 \end{equation}
is the {\it  (quasi-static) average heat absorption} due to the  thermodynamic operation
$\mu_1 \to  \mathcal{L}_{\log J_2}^*(\mu_1)$. The value $\Delta Q$ expresses the change of heat.

\eqref{troto} corresponds  to the expression
\begin{equation} \label{trotoi} \Delta Q= \sum_i E_i (P\, p)_i - \sum_i E_i p_i,
\end{equation}
described at the beginning of section 4.1 in \cite{Sa}, where $E_i, i=1,2..,k$, is energy, $p=(p_1,p_2,...,p_k)$ a  probability and $P$ a stochastic matrix.

Given the H\"older Jacobian $\log J_2$  and  an invariant probability $\mu_1$, denote $\mu_3=\mathcal{L}^*_{\log J_2}(\mu_1)$ and $J_3$
the Radon-Nikodin derivative of $\mu_3$. Remember that $\log J_3 (x) = \log J_1( \sigma(x)) + \log  J_2 (x) - \log J_2 (\sigma(x)).$

The {\it variation of energy} for  the joint system $\mu_1$ and $\mu_2$,   due to the  thermodynamic operation $\mu_1 \to  \mathcal{L}_{\log J_2}^*(\mu_1)$,  is
$$ \Delta U = E(\mu_1)  - E(\mathcal{L}_{\log J_2}(\mu_1))= E(\mu_1)  - E(\mu_3)=$$
\begin{equation} \label{energ1}  [\int \log J_1 d \mu_1 - \int \log J_2 d \mu_1 ]-
[\int \log J_1 d \mu_3 - \int \log J_2 d \mu_3 ] .
\end{equation}
\smallskip

Given the joint system $\mu_1$ and $\mu_2$
the {\it (quasi-static) variation of work}    due to the thermodynamic operation $\mu_1 \to  \mathcal{L}_{\log J_2}(\mu_1)$ is
$$\Delta W(\mu_1,\mu_2)= \int  \log J_1 d \mu_1 -  \int  \log J_1 d \mathcal{L}_{\log J_2}^*(\mu_1)= $$
\begin{equation} \label{wo}  \int  \log J_1 d \mu_1 -  \int  \log J_1 d \mu_3 \geq 0.
\end{equation}
\smallskip

A simple computation shows the {\it First Law of Thermodynamics}:

\begin{proposition} \label{oirt} Conservation of energy
\begin{equation} \label{worte} \Delta W + \Delta Q = \Delta U.
\end{equation}
\end{proposition}

\smallskip

The claim of the above lemma corresponds to \eqref{t2} in Section \ref{gaz}.



See the role of conservation of energy in expression (5.6) in \cite{MNS}.







 Related results for  continuous-time Markov chains are presented in  section 3.1 in \cite{Wang}.

\smallskip

\section{Entropy production via reversion of direction on the lattice} \label{EP}

Entropy production is a very important topic in the study of Thermodynamic of gases (see \cite{MN}, \cite{Cat} and \cite{Altaner}).
Entropy production  is the amount of entropy which is produced in any irreversible processes such as heat and mass transfer, heat exchange, fluid flow, etc., in thermal machines. There are important issues related to the increase of entropy if the system is reversible or irreversible.
One of our purposes on our paper was to describe a large range of topics that could establish  a common ground for discussion between mathematicians and physicists.  This topic was not very much discussed in the Thermodynamic Formalism community. The use of the so called {\it involution kernel} to detect if a given system is reversible or irreversible is not very well known tool for both communities. Below we will elaborate on the topic.

In this section, we consider the classic concept of entropy production that is related to symmetry (or non-symmetry) in the lattice $ \mathbb {Z}$.
The results in this section are not new, they are just a summary of what was presented in \cite{LM}. For a given potential
$ A: \{1,2, ..., d \}^\mathbb{N} \to \mathbb {R} $, we will describe the meaning of being symmetrical with respect to the involution kernel. This can be interpreted as saying that the system associated with this potential is reversible. We will show that the entropy production is zero if the potential is symmetrical with respect to the involution kernel.  Related results appear in \cite{MNS} and section 5 in Chapter II of \cite{Liggett} which consider different settings. 

Assume elements in $\hat{\Omega}= \{1,2,...,d\}^\mathbb{Z}$ will be written in the form
$$(...,y_{3},y_{2},y_1|x_1,x_2,x_3,...)=(y\,|\, x).$$

We point out that in \cite{LM} a more general analysis of KL divergence and entropy production is considered: instead of
$\hat{\Omega}= \{1,2,...,d\}^\mathbb{Z}$ it is considered the case when  $\hat{\Omega}= M^\mathbb{Z}$ and where $M$ is a compact metric space.

In this section we consider the thermodynamic operation $\theta$ which is reversion of time on  $\{1,2,...,d\}^\mathbb{Z}$.
By this we mean $\theta((...,z_{3},z_{2},z_1|) = |z_1,z_2,z_3,...).$
This formulation will produce  the more well-known form of the definition of entropy production. The concept of reversibility or irreversibility is characterized by the entropy production to be, respectively,  zero or not zero. These are fundamental concepts in nonequilibrium Statistical Mechanics (see \cite{MNS}).

Consider a given potential $A$ which is not normalized and the associated equilibrium probability $\mu_A$. We ask if there is a simple way to check if the entropy production of $\mu_A$ is zero. This can be achieved via the use of the concept of Involution Kernel (see \cite{BLT}) and Proposition \ref{variational2}.

We denote $\Omega$ by $\Omega^{+}$. The elements of  $\Omega^{+}$ are denoted by $x=|\,x_1,x_2,...).$
Consider  the space $\Omega^-=\{1,2,...,d\}^{\mathbb{N}}$ (which is formally different from   $\Omega^{+}$).
Any point in the space  $\Omega^-$ will be written in the form $y=(...,y_{3},y_{2},y_1|$ and any point in $\hat{\Omega}=\Omega^-\times \Omega^+$ will be written in the form $(...,y_{3},y_{2},y_1|x_1,x_2,x_3,...)=(y\,|\, x)$.

We consider on $\hat{\Omega}$ the shift map $\hat{\sigma}$ given by
\begin{equation} \label{ret}
\hat{\sigma}((...,y_{3},y_{2},y_1|x_1,x_2,x_3,...)) = (...,y_{3},y_{2},y_1,x_1|x_2,x_3,...).
\end{equation}
The natural restriction of  $\hat{\sigma}$ over $\Omega=\Omega^{+}$ is the shift map $\sigma$. The natural restriction of $\hat{\sigma}^{-1}$ over $\Omega^-$ is denoted by $\sigma^-$. Observe that $(\Omega^-,\sigma^-)$  can be identified with $(\Omega^{+},\sigma)$, via  the conjugation $\theta: \Omega^-\to \Omega^{+}=\Omega$, which is  given by
\begin{equation} \label{ket} \theta((...,z_{3},z_{2},z_1|) = |z_1,z_2,z_3,...). \end{equation}

Any $\sigma$-invariant  probability $\mu$ on $\Omega^+$ can be extended to a $\hat{\sigma}$-invariant probability $\hat{\mu}$ on $\Omega^-\times \Omega^{+}$. The restriction of $\hat{\mu}$ to $\Omega^{-}$, denoted by $\mu^{-}$, is  $\sigma^-$-invariant. By identifying $(\Omega^{-},\sigma^-)$ with $(\Omega,\sigma)$, via the conjugation $\theta$ and denoting by $\theta_*\mu^-$ the push forward of $\mu^-$, we get
\begin{equation}\label{eq1}
\theta_*\mu^-(|a_1,a_2....a_m]) = \mu(|a_m,...,a_{2},a_1]).
\end{equation}

\smallskip

The role $\mu \to \theta^* (\mu)$ can be seen as a thermodynamic operation (reversion of time). We will define the associated entropy production and we will ask for each kind of H\"older potentials $A: \Omega^+ \to \mathbb{R}$ the corresponding Gibbs state
has entropy production zero. This can be characterized via the involution kernel.

\smallskip

\begin{definition} \label{poupou} We say that $A^{-}: \Omega^{-} \to \mathbb{R}$ is the dual potential of $A$ if for some $W$ we get
\[A^{-}(y)=A^{-}((...,y_3,y_2,y_1|) = A(|y_1,x_1,x_2,...) \,)+ W(...,y_3,y_2|y_1,x_1,x_2,x_3,...)\]
\begin{equation}\label{Aminus1}-   W(...,y_3,y_2,y_1|x_1,x_2,x_3,...),\end{equation}
for any $(...,y_3,y_2,y_1|x_1,x_2,x_3,...) \in \hat{\Omega}$.

The important point here is that  $A^{-}$ is a function only  of the variable $y \in \Omega^{-}.$
We call {\it involution kernel}   any $W$ satisfying the above.  
\end{definition}

Given $A$, the involution kernel $W$ is not unique (and therefore the dual potential $A^{-}$ is also not unique.)

If \eqref{Aminus1} is true we say that $W$ is the involution kernel for the dual pair $A,A^{-}$.

The proof that when $A$ is a H\"older function
such $W$ and $A^{-}$ exist appeared initially in \cite{BLT} (for a simple proof of this result see section 4 in \cite{LNotes}).
One can show that $A$ is H\"older and $W$ is bi H\"older.

We say that the potential $A$ is {\it symmetric}  if there exists an involution kernel $W$ such that   $A^{-}=A$. There are many examples of potentials that are symmetric (see \cite{CL3} or section 5 in \cite{BLLbook}). The potential to be symmetrical is related with the detailed balance condition and the entropy production to be zero (see the  claim  of Theorem \ref{variational2} and the  subsequent paragraph).

 \smallskip

For a given H\"older potential $A$ the probability $\mu_A$ denotes the equilibrium probability for $A$.

 \smallskip

\begin{definition}\label{muminuseq} Let $A:\Omega^+\to\mathbb{R}$ be a H\"older function and  $W$ be any bi H\"older involution kernel for $A$. Now, consider the function $A^{-}$ on $\Omega^{-}$ as defined by (\ref{Aminus1}).  We denote
 	$\mu_A^-$ the equilibrium probability for $A^-$ in $\Omega^-$.
 	
\end{definition}

$\mu_A^-$ and $\mu_A$ are dual probabilities associated with a pair of  dual potentials.

Following   \cite{LM} we define:

\begin{definition} Given the potential $A$, the {\bf $A$-flip entropy production} of the  equilibrium probability $\mu_A$ is defined as
	\begin{equation} \label{bact45} e_p(\mu_A) = D_{KL}(\mu_A,\theta_*\mu^-_A),
	\end{equation}
	where $\theta_*\mu^-_A$ on $\Omega^{+}$ is the push-forward of $\mu^{-}_A$ by the conjugation $\theta:\Omega^-\to\Omega^+$ given by \eqref{ket}.
\end{definition}	
\smallskip

\begin{theorem}\label{variational2} Suppose that $\mu_A$ is the equilibrium probability for the H\"older function $A:\Omega^+\to\mathbb{R}$. Let $W$ be any H\"older involution kernel for $A$ and $A^-:\Omega^-\to \mathbb{R}$ be the dual  function. Suppose that $A^{-}$ is defined on $\Omega^+$ using the conjugation $\theta$. Then, the entropy production is
\begin{equation} \label{tritro} e_p(\mu_A) =  \int A - A^-\,d\mu_A .
\end{equation}
\end{theorem}
\smallskip
For the proof of the above theorem see \cite{LM}.
\smallskip

Suppose we just want to show that $ \mu_A $ has zero entropy production, that is, $ e_p (\mu_A) = 0. $
To calculate $ e_p (\mu_A) $, we need to have explicit information about $  \mu_A $ and $ \mu^{ -}_ A $. Given a non-normalized potential $ A $, in order to have explicit information about the probability of equilibrium $ \mu_A $ it is necessary to find the eigenvalue and eigenfunction of the Ruelle operator
$ \mathcal {L} _A $ (see expression \eqref {pois}). 
 This is clearly a technical difficulty; among other things, you need to guess the exact value of the eigenvalue.  
Note that given $A$, to solve  \eqref{Aminus1} is a much more simple problem because you do not need to find the eigenvalue. The next result (a more simple form of solving the problem) follows immediately from Theorem \ref {variational2}:

\begin{corollary}  If we are able to get an involution kernel $W$, such that, the  dual potential $A^{-}$ is equal to $A$, then the entropy production is equal to zero.

\end{corollary}

\begin{remark} \label{poix} 
 Given a non normalized H\"older potential $A$, we say that the probability   $\nu_A$ on $\Omega$  is the eigenprobability for the dual of the Ruelle operator  $\mathcal{L}_A^*$, if
 $\mathcal{L}_A^*(\nu_A) =\lambda_A\, \nu_A$, where $\lambda_A$ is the main eigenvalue for $\mathcal{L}_A.$ Note that from \cite{BLT} (or \cite{LMMS}) the main eigenvalue of the Ruelle operator for $A$ and the  main eigenvalue of the Ruelle operator for its dual $A^{-}$ are the same. The same thing for $\mathcal{L}_A^*$ and $\mathcal{L}_{A^{-}}^*.$

An interesting property (see \cite{BLT}) of the involution kernel $W$ is the following type of duality: assume $A^{-}$ the dual of $A$   and $\nu_{A^{-}}$ is the eigenprobability for the dual of the Ruelle operator  $\mathcal{L}_{A^{-}}^*$, then,
\begin{equation} \label{tritri} \int e^{ W (y|x)} d \nu_{A^{-}}(y) = \varphi_A(x)
\end{equation}
is the main eigenfunction of the Ruelle operator $\mathcal{L}_A$, where 
$W$ is the involution kernel for the dual pair $A$, $ A^{-}$. Other eigenfunctions (not strictly positive) can be eventually obtained via 
eigendistributions for $\mathcal{L}_A^*$ (see \cite{GLP}).
\end{remark}

\end{document}